\documentclass[12pt]{amsart}

\usepackage[colorlinks=true,linkcolor=blue]{hyperref}
\usepackage{amsmath}
\usepackage{amssymb}
\usepackage{amsthm}
\usepackage{bm} 

\usepackage[T1]{fontenc} 

\usepackage{graphicx}
\usepackage{epstopdf}
\usepackage{epsfig}

\usepackage{tikz}
\usepackage{tikz-cd}

\theoremstyle{plain} 
\newtheorem{thm}{Theorem}[section]

\newtheorem{prop}[thm]{Proposition}
\newtheorem{lemma}[thm]{Lemma}

\theoremstyle{remark}
\newtheorem{remark}[thm]{Remark}

\newtheorem{example}[thm]{Example}

\theoremstyle{definition}
\newtheorem{defin}[thm]{Definition}

\newcommand{\PP}{\mathbb{P}}
\newcommand{\QQ}{\mathbb{Q}}

\newcommand{\ZZ}{\mathbb{Z}}

\newcommand{\calO}{{\mathcal O}}

\newcommand{\Lbar}{\overline{L}}
\newcommand{\Kbar}{\overline{K}}

\newcommand{\PKbar}{\PP^1(\Kbar)}

\DeclareMathOperator{\sgn}{sgn}
\DeclareMathOperator{\dist}{dist}

\DeclareMathOperator{\Orb}{Orb}

\DeclareMathOperator{\charact}{char}

\DeclareMathOperator{\Gal}{Gal}

\DeclareMathOperator{\Aut}{Aut}

\newcommand{\fp}{ {\mathfrak p} }

\newcommand{\fP}{ {\mathfrak P} }

\newcommand{\dsps}{\displaystyle}

\begin{document}

\title[Arboreal Galois groups for cubic polynomials]
{Arboreal Galois groups for cubic polynomials
with colliding critical points}
\date{April 4, 2024}
\subjclass[2010]{37P05, 11R32, 14G25}
\author[Benedetto]{Robert L. Benedetto}
\email{rlbenedetto@amherst.edu}
\address{Amherst College \\ Amherst, MA 01002 \\ USA}
\author[DeGroot]{William DeGroot}
\email{william.h.degroot.gr@dartmouth.edu}
\address{Dartmouth College \\ Hanover, NH 03755 \\ USA}
\author[Ni]{Xinyu Ni}
\email{xni23@amherst.ed}
\address{Amherst College \\ Amherst, MA 01002 \\ USA}
\author[Seid]{Jesse Seid}
\email{jseid23@amherst.edu}
\address{Amherst College \\ Amherst, MA 01002 \\ USA}
\author[Wei]{Annie Wei}
\email{aw804@scarletmail.rutgers.edu}
\address{Rutgers University \\ New Brunswick, NJ 08901 \\ USA}
\author[Winton]{Samantha Winton}
\email{swinton23@amherst.edu}
\address{Amherst College \\ Amherst, MA 01002 \\ USA}

\begin{abstract}
Let $K$ be a field, and let $f\in K(z)$ be a rational function of degree $d\geq 2$.
The Galois group of the field extension generated by the preimages of $x_0\in K$
under all iterates of $f$ naturally embeds in the automorphism group of
an infinite $d$-ary rooted tree. In some cases the Galois group can be the
full automorphism group of the tree, but in other cases it is known to have
infinite index. In this paper, we consider a previously unstudied such case:
that $f$ is a polynomial of degree $d=3$,
and the two finite critical points of $f$ collide at the $\ell$-th iteration,
for some $\ell\geq 2$. We describe an explicit subgroup $Q_{\ell,\infty}$
of automorphisms of the $3$-ary tree in which the resulting Galois group
must always embed, and we present sufficient conditions for this embedding
to be an isomorphism.
\end{abstract}

\maketitle

\section{Introduction}
Let $f\in K[z]$ be a polynomial of degree $d\geq 2$ defined over a field $K$.
Let $\Kbar$ be an algebraic closure of $K$.
For any integer $n\geq 0$, we write $f^n:=f\circ\cdots\circ f$,
$f^0(z)=z$, so that $f^n\in K[z]$ is a polynomial of degree $d^n$.

Let $x_0\in\Kbar$.
The \emph{forward orbit} of $x_0$ under $f$ is
\[ \Orb_f^{+}(x_0) := \{f^n(x_0) \, | \, n\geq 0\} \subseteq K(x_0), \]
and the \emph{backward orbit} of $x_0$ under $f$ is
\[ \Orb_f^{-}(x_0) := \coprod_{n\geq 0} f^{-n}(x_0) \subseteq \Kbar, \]
where $f^{-n}(y):=(f^n)^{-1}(y)$ denotes the set of $d^n$ solutions
(counted with multiplicity) of the equation $f^n(z)=y$ in $\Kbar$.

In this paper, we study the fields generated by the backward orbit of $x_0\in K$.
More precisely, 
for each $n\geq 0$, define
\[ K_n:=K(f^{-n}(x_0))\subseteq\Kbar, \quad\text{and}\quad
K_{\infty}:=\bigcup_{n\geq 0} K_n\subseteq\Kbar. \]
If $f$ is separable (as a mapping from $\PKbar$ to itself, which is certainly true if $\charact K=0$, for example),
then each $K_n/K$ is a separable and hence Galois extension. We define
\[ G_{K,n}:=\Gal(K_n/K) \quad\text{and}\quad
G_{K,\infty}:=\Gal(K_{\infty}/K)\cong\varprojlim G_{K,n} \]
to be the associated Galois groups over $K$.

If $f$ has no critical points in the backward orbit of $x_0$,
then we may consider the elements of $\Orb_f^-(x_0)$ as the nodes
of an infinite $d$-ary rooted tree $T_{d,\infty}$.
Here, $x_0$ is the root node, and the $d^n$ nodes of the tree at level~$n$
correspond to the $d^n$ elements of $f^{-n}(x_0)$;
furthermore, for each $n\geq 1$, we connect $y\in f^{-n}(x_0)$ to $f(y)\in f^{-(n-1)}(x_0)$ by an edge.
Since $f$ is defined over $K$, any $\sigma\in G_{K,\infty}$ must preserve this tree structure,
and hence $G_{K,\infty}$
is isomorphic to a subgroup of the automorphism group $\Aut(T_{d,\infty})$ of the tree.
Similarly, for any $n\geq 0$, if $T_{d,n}$ denotes the rooted $d$-ary tree up only to level~$n$,
then $G_{K,n}$ is isomorphic to a subgroup of $\Aut(T_{d,n})$.
(Here and throughout this paper, when we say that two groups
acting on a tree are isomorphic, we mean that the isomorphism is
equivariant with respect to the action on the tree.)

The study of such Galois groups was initiated in 1985 by Odoni \cite{Odoni},
later dubbed \emph{arboreal} Galois representations by Boston and Jones in 2007 \cite{BosJon}.
Much work in this area has centered on proving that $G_{K,\infty}$ can be the full group $\Aut(T_{d,\infty})$,
at least when $K$ is a number field or function field; see, for example,
\cite{BenJuu,Juul,Kadets,Looper,Odoni,Specter,Stoll}.
(For more general fields there is no such expectation,
even over some Hilbertian fields, as shown in \cite{DK22}.)
Indeed, by analogy with Serre's open index theorem \cite{Serre}
for Galois representations arising from elliptic curves, a folklore conjecture
predicts that when $K$ is a number field or function field,
$G_{K,\infty}$ should usually have finite index in $\Aut(T_{d,\infty})$.
See \cite[Conjecture~3.11]{Jones} for a precise version of this conjecture in degree $d=2$.
See also \cite{BGJT23,BEK21,FM,FPC,Hindes2,Ing13,JonMan,JKMT,Sing23,Swam}
for other results on arboreal Galois groups.

However, Serre's Open Image Theorem makes an exception for CM curves,
and similarly there are known situations when $G_{K,\infty}$ is of infinite index in $\Aut(T_{d,\infty})$,
even over number fields and function fields.
See, for example, the results of \cite{BHL} for maps with certain extra symmetries,
or \cite{ABCCF,BFHJY,Pink} for maps that are 
\emph{postcritically finite}, or PCF,
meaning that every critical point $c$ is preperiodic.
In addition, if the root point $x_0$ is periodic or in the forward orbit of a critical value,
it is easy to check that the resulting restrictions on the backward orbit of $x_0$ ensure that
$[\Aut(T_{d,\infty}) : G_{K,\infty}]=\infty$. Another property that can force
this index to be infinite is the following; we have stated it for the more general setting
of rational functions rather than polynomials.

\begin{defin}
\label{def:collide}
Let $f\in K(z)$ be a rational function,
let $\xi_1,\xi_2\in\PKbar$ be two critical points of $f$,
and let $\ell\geq 1$ be a positive integer.
We say that $\xi_1$ and $\xi_2$ \emph{collide}
at the $\ell$-th iterate if
\begin{equation}
\label{eq:PinkCond}
f^{\ell}(\xi_1)=f^{\ell}(\xi_2)
\quad\text{but } f^{\ell-1}(\xi_1)\neq f^{\ell-1}(\xi_2).
\end{equation}
\end{defin}

If each of the critical points of $f$ either collides with one other or is itself preperiodic,
it can be shown that $[\Aut(T_{d,\infty}) : G_{K,\infty}]=\infty$.
Pink was the first to note this, in \cite[Theorem~4.8.1]{Pink2},
for rational functions of degree~2 (and assuming $\charact K\neq 2$),
which have only two critical points.
The first author and Dietrich reformulated Pink's result for quadratic rational maps
with colliding critical points in \cite{BenDie}.

In this paper, we consider the analogous for cubic polynomials $f$ whose two finite critical
points collide at the $\ell$-th iterate. That is, two of the four critical points of $f$ in $\PKbar$
are at the (fixed) point at $\infty$ and hence are preperiodic, while the other two collide
with one another. Note that $\ell\geq 2$ necessarily, because if $\ell$ were $1$, the
critical image $f(\xi_1)=f(\xi_2)$ would have two preimages at each of $\gamma_1$
and $\gamma_2$, counting multiplicity, and the resulting total of four is too many
for a cubic map.

We define two subgroups
$Q_{\ell,\infty}\subseteq\widetilde{Q}_{\ell,\infty}$ of infinite index in $\Aut(T_{3,\infty})$
associated to a cubic polynomial whose two finite critical points collide at the $\ell$-th iteration;
see Definition~\ref{def:QGroup}. Our first main result is then to show that the resulting
arboreal group is necessarily contained in $\widetilde{Q}_{\ell,\infty}$, 
or even in $Q_{\ell,\infty}$, depending on whether or not $-3$ is a square in $K$.

\begin{thm}
\label{thm:main1}
Let $K$ be a field of characteristic not dividing $6$.
Let $f\in K[z]$ be a cubic polynomial with critical points $\gamma_1,\gamma_2\in\Kbar$
that collide at the $\ell$-th iterate, for some integer $\ell\geq 2$.
Then $f^{\ell}(\gamma_1)=f^{\ell}(\gamma_2)$ is $K$-rational.

Furthermore, fix any $x_0\in K$, and let $G_{K,\infty}=\Gal(K_{\infty}/K)$ be the arboreal Galois group for
$f$ over $K$, rooted at $x_0$. Then:
\begin{enumerate}
\item $G_{K,\infty}$ is isomorphic to a subgroup of $\widetilde{Q}_{\ell,\infty}$,
via an appropriate labeling of the tree.
\item $G_{K,\infty}$ is isomorphic to a subgroup of $Q_{\ell,\infty}$
if and only if $-3$ is a square in $K$.
\end{enumerate}
\end{thm}

In our second main result, we present sufficient conditions for such a cubic polynomial $f$
and root point $x_0\in K$ to have arboreal Galois group $G_{K,\infty}$ equal to all of 
$\widetilde{Q}_{\ell,\infty}$ or $Q_{\ell,\infty}$. To state it,
for a cubic polynomial $f\in K[z]$ with critical points
$\gamma_1,\gamma_2\in\Kbar$, define $A\in K^{\times}$ to be the lead coefficient of $f$,
and define $B,C_1,C_2,\ldots \in K$ by
\begin{equation}
\label{eq:BCdef}
B := f'\bigg(\frac{1}{2}(\gamma_1+\gamma_2)\bigg)
\quad\text{and}\quad
C_n := \big( f^n(\gamma_1) - f^n(\gamma_2) \big)^2
\end{equation}
Note that $B,C_1,C_2,\ldots$
are indeed $K$-rational because they lie in $K(\gamma_1)=K(\gamma_2)$
and are $\Gal(K(\gamma_1)/K)$-invariant.


\begin{thm}
\label{thm:main2}
With notation as in Theorem~\ref{thm:main1} and equation~\eqref{eq:BCdef},
suppose further that $K$ is the field of fractions of a Dedekind domain $\calO_K$.
Also suppose that there is a sequence of pairwise distinct places
$v_1,v_2,\ldots\in M_K^0$
such that
\begin{enumerate}
\item $v_n(A)=v_n(B)=v_n(6)=0 \leq v_n(x_0)$
for every $n\geq 1$,
\item $v_n(C_j)=0$ for all $j=1,\ldots,\ell-1$ and all $n\geq j$,
\item $v_n((f^i(\gamma_1)-x_0)(f^i(\gamma_2)-x_0))=0$ for all $i,n$
with $1\leq i\leq n -1$,
\item $v_n((f^n(\gamma_1)-x_0)(f^n(\gamma_2)-x_0))$ is odd for all $1\leq n\leq \ell -1$, and
\item $v_n(f^n(\gamma_1)-x_0)$ is odd for all $n\geq \ell$.
\end{enumerate}
Suppose also that there is a non-archimedean place $u$ of $K$
such that $u(A)=0\leq u(B)$, and for which $u(x_0)$ is negative and prime to $3$.
Then 
$G_{K,\infty}\cong\widetilde{Q}_{\ell,\infty}$ if $\sqrt{-3}\not\in K$,
and 
$G_{K,\infty}\cong Q_{\ell,\infty}$ if $\sqrt{-3}\in K$.
\end{thm}

Here $M_K^0$ denotes the set of non-archimedean places of $K$, each normalized to have image equal to $\ZZ$.
That is, each $v\in M_K^0$ is the valuation $v=v_{\fp}$ of a prime ideal $\fp$ of the ring of integers $\calO_K$,
normalized so that $v_{\fp}(\pi)=1$, where $\pi\in \calO_K$ is a uniformizer for $\fp$.
Alternatively, if $K$ is a function field, we allow $v$ to be a place at $\infty$,
such as the negative degree valuation.

The outline of the paper is as follows.
In Section~\ref{sec:prelim}, after presenting our notation and describing labelings
of the tree $T_{3,\infty}$, we define the sign $\sgn_n(\sigma,y)\in\{\pm 1\}$
to describe the parity of $\sigma\in\Aut(T_{3,\infty})$ acting on the $3^n$ nodes sitting $n$ levels above
the node $y$. We then use $\sgn_n$ to define
our groups $Q_{\ell,\infty}$ and $\widetilde{Q}_{\ell,\infty}$.
In Section~\ref{sec:collideQ}, we prove several results
about the signs $\sgn_n$ and about iterated discriminants $\Delta(f^n-x_0)$,
and we use the resulting formulas to prove Theorem~\ref{thm:main1}.
Section~\ref{sec:surj} is devoted to group-theoretic results needed
to prove Theorem~\ref{thm:main2}.
In Section~\ref{sec:tech}, we reduce to the case that $f$ is of the form
$f(z)=Az^3+Bz+1$, and we observe in Remark~\ref{rem:BC} that this use of $B$
agrees with the quantity in equation~\eqref{eq:BCdef} above.
Also in Section~\ref{sec:tech}, we prove a number of elementary computational
formulas for the coefficients of various polynomials needed in the proof of Theorem~\ref{thm:main2}.
We then combine these formulas with the group theory of Section~\ref{sec:surj}
to prove Theorem~\ref{thm:main2} in Section~\ref{sec:galproof}.
Finally, in Section~\ref{sec:ex}, we present a few examples.

\section{Preliminaries}
\label{sec:prelim}

We set the following notation throughout this paper.
\begin{tabbing}
\hspace{8mm} \= \hspace{15mm} \=  \kill
\> $K$: \> a field of characteristic different from $2$ and $3$,
with algebraic closure $\Kbar$ \\
\> $\ell$: \> an integer $\ell\geq 2$ \\
\> $f$: \> a cubic polynomial $f(z) \in K[z]$ \\
\> $x_0$: \> an element of $K$, to serve as the root of our preimage tree \\
\> $T_{3,n}$: \> a ternary rooted tree, extending $n$ levels above its root node \\
\> $T_{3,\infty}$: \> a ternary rooted tree, extending infinitely above its root node \\
\> $K_n$: \> for each $n\geq 0$, the extension field $K_n:=K(f^{-n}(x_0))$ \\
\> $K_\infty$: \> the union $K_{\infty} = \bigcup_{n\geq 1} K_n$ in $\Kbar$ \\
\> $G_{K,n}$: \> the Galois group $\Gal(K_n/K)$ \\
\> $G_{K,\infty}$: \> the Galois group $\Gal(K_\infty/K)$
\end{tabbing}

A \emph{labeling} of the tree $T_{3,n}$ or $T_{3,\infty}$ is an assignment of
a unique label $s_1 s_2 \ldots s_m$ of $m$ symbols $s_i\in\{0,1,2\}$ to each node $y$
at level $m$ of the tree, in such a way that the node immediately below $y$ has label
$s_1 s_2 \ldots s_{m-1}$. (We assign the empty label $()$ to the root node.)
See Figure~\ref{fig:labeltree}.

%
%
%

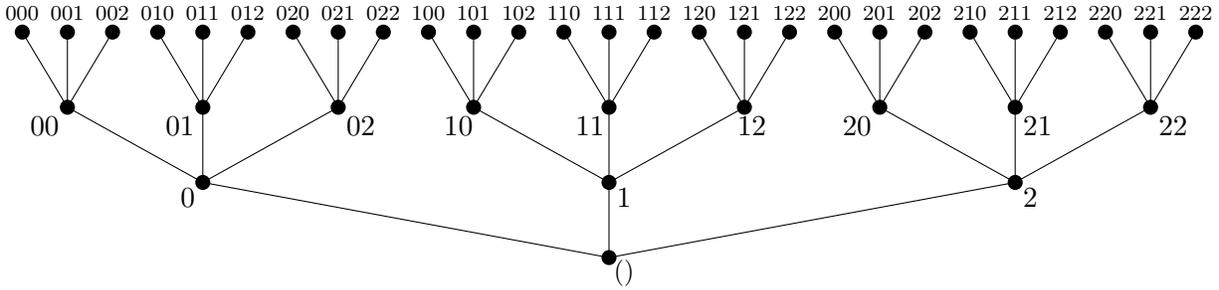
\begin{figure}
\begin{tikzpicture}
\path[draw] (0,0) -- (0,3);
\path[draw] (-5.4,3) -- (-5.4,1) -- (0,0) -- (5.4,1) -- (5.4,3);
\path[draw] (-7.2,3) -- (-7.2,2) -- (-5.4,1) -- (-3.6,2) -- (-3.6,3);
\path[draw] (-1.8,3) -- (-1.8,2) -- (0,1) -- (1.8,2) -- (1.8,3);
\path[draw] (7.2,3) -- (7.2,2) -- (5.4,1) -- (3.6,2) -- (3.6,3);
\path[fill] (0,0) circle (0.1);
\path[fill] (0,1) circle (0.1);
\path[fill] (0,2) circle (0.1);
\path[fill] (0,3) circle (0.1);
\path[fill] (-5.4,1) circle (0.1);
\path[fill] (-5.4,2) circle (0.1);
\path[fill] (-5.4,3) circle (0.1);
\path[fill] (5.4,1) circle (0.1);
\path[fill] (5.4,2) circle (0.1);
\path[fill] (5.4,3) circle (0.1);
\path[fill] (-7.2,2) circle (0.1);
\path[fill] (-7.2,3) circle (0.1);
\path[fill] (7.2,2) circle (0.1);
\path[fill] (7.2,3) circle (0.1);
\path[fill] (-3.6,2) circle (0.1);
\path[fill] (-3.6,3) circle (0.1);
\path[fill] (3.6,2) circle (0.1);
\path[fill] (3.6,3) circle (0.1);
\path[fill] (-1.8,2) circle (0.1);
\path[fill] (-1.8,3) circle (0.1);
\path[fill] (1.8,2) circle (0.1);
\path[fill] (1.8,3) circle (0.1);
\path[draw] (7.8,3) -- (7.2,2) -- (6.6,3);
\path[draw] (6,3) -- (5.4,2) -- (4.8,3);
\path[draw] (4.2,3) -- (3.6,2) -- (3,3);
\path[draw] (2.4,3) -- (1.8,2) -- (1.2,3);
\path[draw] (-.6,3) -- (0,2) -- (.6,3);
\path[draw] (-7.8,3) -- (-7.2,2) -- (-6.6,3);
\path[draw] (-6,3) -- (-5.4,2) -- (-4.8,3);
\path[draw] (-4.2,3) -- (-3.6,2) -- (-3,3);
\path[draw] (-2.4,3) -- (-1.8,2) -- (-1.2,3);
\path[fill] (7.8,3) circle (0.1);
\path[fill] (6.6,3) circle (0.1);
\path[fill] (6,3) circle (0.1);
\path[fill] (4.8,3) circle (0.1);
\path[fill] (4.2,3) circle (0.1);
\path[fill] (3,3) circle (0.1);
\path[fill] (2.4,3) circle (0.1);
\path[fill] (1.2,3) circle (0.1);
\path[fill] (.6,3) circle (0.1);
\path[fill] (-7.8,3) circle (0.1);
\path[fill] (-6.6,3) circle (0.1);
\path[fill] (-6,3) circle (0.1);
\path[fill] (-4.8,3) circle (0.1);
\path[fill] (-4.2,3) circle (0.1);
\path[fill] (-3,3) circle (0.1);
\path[fill] (-2.4,3) circle (0.1);
\path[fill] (-1.2,3) circle (0.1);
\path[fill] (-.6,3) circle (0.1);
\node (x0) at (0.2,-0.2) {\footnotesize $()$};
\node (a0) at (-5.6,0.8) {\small $0$};
\node (a1) at (0.2,0.8) {\small $1$};
\node (a2) at (5.6,0.8) {\small $2$};
\node (b00) at (-7.5,1.75) {\small $00$};
\node (b01) at (-5.7,1.75) {\small $01$};
\node (b02) at (-3.3,1.75) {\small $02$};
\node (b10) at (-2,1.75) {\small $10$};
\node (b11) at (-.25,1.75) {\small $11$};
\node (b12) at (1.9,1.75) {\small $12$};
\node (b20) at (3.3,1.75) {\small $20$};
\node (b21) at (5.7,1.75) {\small $21$};
\node (b22) at (7.5,1.75) {\small $22$};
\node (c000) at (-7.8,3.25) {\tiny $000$};
\node (c001) at (-7.2,3.25) {\tiny $001$};
\node (c002) at (-6.6,3.25) {\tiny $002$};
\node (c010) at (-6.0,3.25) {\tiny $010$};
\node (c011) at (-5.4,3.25) {\tiny $011$};
\node (c012) at (-4.8,3.25) {\tiny $012$};
\node (c020) at (-4.2,3.25) {\tiny $020$};
\node (c021) at (-3.6,3.25) {\tiny $021$};
\node (c022) at (-3,3.25) {\tiny $022$};
\node (c100) at (-2.4,3.25) {\tiny $100$};
\node (c101) at (-1.8,3.25) {\tiny $101$};
\node (c102) at (-1.2,3.25) {\tiny $102$};
\node (c110) at (-0.6,3.25) {\tiny $110$};
\node (c111) at (0,3.25) {\tiny $111$};
\node (c112) at (0.6,3.25) {\tiny $112$};
\node (c120) at (1.2,3.25) {\tiny $120$};
\node (c121) at (1.8,3.25) {\tiny $121$};
\node (c122) at (2.4,3.25) {\tiny $122$};
\node (c200) at (3,3.25) {\tiny $200$};
\node (c201) at (3.6,3.25) {\tiny $201$};
\node (c202) at (4.2,3.25) {\tiny $202$};
\node (c210) at (4.8,3.25) {\tiny $210$};
\node (c211) at (5.4,3.25) {\tiny $211$};
\node (c212) at (6,3.25) {\tiny $212$};
\node (c220) at (6.6,3.25) {\tiny $220$};
\node (c221) at (7.2,3.25) {\tiny $221$};
\node (c222) at (7.8,3.25) {\tiny $222$};
\end{tikzpicture}
\caption{Labeling the tree $T_{3,3}$}
\label{fig:labeltree}
\end{figure}

The backward orbit $\Orb^-_f(x_0)$ of $x_0\in K$ has a natural structure of $T_{3,\infty}$,
with $y\in f^{-m}(x_0)$ corresponding to a node at level $m$ of the tree, connected to $f(y)\in f^{-(m-1)}(x_0)$
at level $m-1$.
Having fixed a labeling of this tree,
we will often abuse notation and refer to the value $y\in f^{-m}(x_0)$
interchangeably with its label $s_1s_2\ldots s_m$.
(We will generally assume that $x_0$ is not periodic, and that there are no critical points in
the backward orbit of $x_0$, so that no two different nodes of the tree correspond to the same element of $\Kbar$.)


Having fixed a labeling of the tree $T_{3,\infty}$,
then for any node $y$ of the tree and any positive integer $m\geq 1$,
the $3^m$ nodes that are $m$ levels above $y$ have labels
$y s_1 s_2 \ldots s_m$, with each $s_i\in\{0,1,2\}$.
For any automorphism $\sigma\in\Aut(T_{3,\infty})$ of the (rooted) tree,
we have
\[ \sigma(y s_1 s_2 \ldots s_m) = \sigma(y) t_1 t_2 \ldots t_m,
\quad\text{for some } t_1,\ldots,t_m\in\{0,1,2\}. \]
Thus, $\sigma$ and $y$ together induce a bijective function from $\{0,1,2\}^m$ to itself,
sending $(s_1, \ldots, s_m)$ to $(t_1, \ldots, t_m)$.
We define the \emph{$m$-th sign} of $\sigma$ above $y$,
denoted $\sgn_m(\sigma,y)$, to be the sign
of this permutation of  $\{0,1,2\}^m$
--- that is, $+1$ if the permutation is of even parity, or $-1$ if it is odd.
Note that if $\sigma(y)\neq y$, then the value of $\sgn_m(\sigma,y)$
depends on the choice of labeling, which is why we fixed a labeling in advance.

\begin{defin}
\label{def:QGroup}
Fix a labeling of the tree $T_{3,\infty}$. Let $\ell\geq 2$ be an integer.
We define $\widetilde{Q}_{\ell,\infty}$ to be the set of all $\sigma\in\Aut(T_{3,\infty})$ for which
\[ \sgn_{\ell}(\sigma,y) \sgn_{\ell-1}(\sigma,y)
= \sgn_{\ell}(\sigma,x_0) \sgn_{\ell-1}(\sigma,x_0)
\quad \text{for every node } y \text{ of } T_{3,\infty} .\]
We also define $Q_{\ell,\infty}$ to be the set of all $\sigma\in \widetilde{Q}_{\ell,\infty}$
for which this common product is $+1$.
\end{defin}

\begin{thm}
\label{thm:QGroup}
Fix a labeling of the tree $T_{3,\infty}$ and an integer $\ell\geq 2$.
Then $\widetilde{Q}_{\ell,\infty}$ is a subgroup of $\Aut(T_{3,\infty})$ of infinite index,
and $Q_{\ell,\infty}$ is a subgroup of $\widetilde{Q}_{\ell,\infty}$ of index $2$.
\end{thm}

\begin{proof}
We begin by claiming that
for any $\sigma,\tau\in\Aut(T_{3,\infty})$, any node $y$ of $T_{3, \infty}$,
and any $n\geq 1$, we have
\begin{equation}
\label{eq:Parident}
\sgn_{n}(\sigma\tau,y)
= \sgn_{n}(\sigma,\tau(y)) \cdot \sgn_{n}(\tau,y).
\end{equation}
To see this, observe that as $\tau$ maps the nodes $n$ levels above $y$ to those above $\tau(y)$,
it permutes the labels with sign $\sgn_{n}(\tau,y)$.
Then $\sigma$ maps those nodes to the nodes above $\sigma(\tau(y))$, permuting their labels
with sign  $\sgn_{n}(\sigma,\tau(y))$. The claim of equation~\eqref{eq:Parident} follows.

We now show that $\widetilde{Q}_{\ell,\infty}$ and $Q_{\ell,\infty}$ are subgroups.
The identity automorphism $e$ clearly belongs to $Q_{\ell,\infty}\subseteq\widetilde{Q}_{\ell,\infty}$.
Given $\sigma,\tau \in \widetilde{Q}_{\ell}$, then for any node $y$ of the tree, equation~\eqref{eq:Parident} yields
\begin{align*}
\sgn_{\ell}(\sigma\tau,y) \sgn_{\ell-1}(\sigma\tau,y) &= 
\sgn_{\ell}(\sigma,\tau(y)) \sgn_{\ell}(\tau,y) \cdot \sgn_{\ell-1}(\sigma,\tau(y)) \sgn_{\ell-1}(\tau,y)
\\
&= \sgn_{\ell}(\sigma,\tau(x_0)) \sgn_{\ell}(\tau,x_0) \cdot \sgn_{\ell-1}(\sigma,\tau(x_0)) \cdot \sgn_{\ell-1}(\tau,x_0)
\\
&= \sgn_{\ell}(\sigma\tau,x_0) \sgn_{\ell-1}(\sigma\tau,x_0),
\end{align*}
where in the second equality we have used the defining property of $\widetilde{Q}_{\ell,\infty}$ to deduce
\[ \sgn_{\ell}(\sigma,\tau(y))\sgn_{\ell-1}(\sigma,\tau(y))
= \sgn_{\ell}(\sigma,x_0)\sgn_{\ell-1}(\sigma,x_0)
= \sgn_{\ell}(\sigma,\tau(x_0))\sgn_{\ell-1}(\sigma,\tau(x_0)) .\]
It follows that $\sigma\tau\in\widetilde{Q}_{\ell,\infty}$.
A similar computation shows that
if $\sigma,\tau\in Q_{\ell,\infty}$, then $\sigma\tau\in Q_{\ell,\infty}$.

Moveover, given any $\sigma\in \widetilde{Q}_{\ell,\infty}$, then choosing
$\tau=\sigma^{-1}\in\Aut(T_{3,\infty})$, equation~\eqref{eq:Parident} yields
\begin{align*}
\sgn_{\ell}(\sigma^{-1},y) \sgn_{\ell-1}(\sigma^{-1},y)
&= \sgn_{\ell}(\sigma, \sigma^{-1}(y))\sgn_{\ell-1}(\sigma, \sigma^{-1}(y))
\\
&= \sgn_{\ell}(\sigma, \sigma^{-1}(x_0))\sgn_{\ell-1}(\sigma, \sigma^{-1}(x_0))
\\
& = \sgn_{\ell}(\sigma^{-1},x_0) \sgn_{\ell-1}(\sigma^{-1},x_0),
\end{align*}
and therefore $\sigma^{-1}\in\widetilde{Q}_{\ell,\infty}$.
Similarly, $Q_{\ell,\infty}$ is also closed under inverses, so both
$\widetilde{Q}_{\ell,\infty}$ and $Q_{\ell,\infty}$ are subgroups of $\Aut(T_{3,\infty})$.

For each integer $m\geq \ell$, pick $\tau_m\in\Aut(T_{3,\infty})$ that acts as a transposition of two
nodes at level $m$ of the tree (and hence also swaps the subtrees rooted at those two nodes)
but otherwise acts as the identity. (That is, there is a node $y_{m}$ at level $m-1$ for which $\tau_m$
swaps the nodes with labels $y_{m} 0 w$ and $y_{m} 1 w$, for any finite word $w$ in the symbols $0,1,2$;
and $\tau(x)=x$ for every other node of the tree.)
Then if $x_m$ is the node $\ell-2$ levels below $y_{m-1}$, we have
$\sgn_{\ell}(\tau_m,x_m)\sgn_{\ell-1}(\tau_m,x_m)=(+1)(-1)=-1$, whereas this product
is $+1$ at most other nodes. It follows quickly that each of $\tau_{\ell}, \tau_{\ell+2}, \tau_{\ell+4},\ldots$
lies in a different coset of $\widetilde{Q}_{\ell,\infty}$.
Therefore, $[\Aut(T_{3,\infty}) : \widetilde{Q}_{\ell,\infty}]=\infty$.

Finally, define $\psi:\widetilde{Q}_{\ell,\infty} \to \{\pm 1\}$ by
$\sigma\mapsto \sgn_{\ell}(\sigma,x_0) \sgn_{\ell-1}(\sigma,x_0)$.
Then $\psi$ is a homomorphism by equation~\eqref{eq:Parident},
and its kernel is $Q_{\ell,\infty}$.
To see that $\psi$ is onto,
define $\rho\in\Aut(T_{3,\infty})$ acting by transposing the labels $0$ and $1$
at every even-numbered level of the tree. Then
$\sgn_{\ell}(\rho,y) \sgn_{\ell-1}(\rho,y)=-1$ for every node $y$ of the tree,
so $\rho\in \widetilde{Q}_{\ell,\infty}$ with $\psi(\rho)=-1$.
Thus, $\psi$ is a surjective homomorphism with kernel $Q_{\ell,\infty}$,
whence $[\widetilde{Q}_{\ell,\infty}:Q_{\ell,\infty}] = |\{\pm 1\}|=2$.
\end{proof}

%

\section{Proving Theorem~\ref{thm:main1}}
\label{sec:collideQ}

\subsection{Iterated discriminants}
\label{ssec:iterdisc}
Our analysis of arboreal Galois groups will require extensive use of discriminants.
Let $g(z)=b_d z^d + \cdots + b_1 z + b_0 \in K[z]$ be a polynomial of degree $d\geq 1$.
Writing $g(z)=b_d \prod_{i=1}^d (z-\beta_i)$ with $\beta_i\in\Kbar$, recall that the discriminant of $g$ is
\[ \Delta(g) := b_d^{2d-2} \prod_{i<j} (\beta_i - \beta_j)^2 \in K  .\]
In the context of arboreal Galois groups, consider a polynomial $f(z)\in K[z]$
of degree $d\geq 2$ with lead coefficient $A\in K^{\times}$, and let $x_0\in K$.
Then for every $n\geq 1$, we have
\begin{equation}
\label{eq:polyiter}
\Delta\big(f^n- x_0 \big)
= (-1)^{d^n (d-1)/2} d^{d^n} A^{d^{2n-1}-1} \big( \Delta(f^{n-1}-x_0)\big)^d
\prod_{f'(\gamma)=0} \big( f^n(\gamma)-x_0 \big),
\end{equation}
where the product is over all finite critical points of $f$,
repeated according to multiplicity.
See, for example, \cite[Remark~3.8]{BenDie}, or
\cite[Proposition~3.2]{AHM}, or \cite[Theorem~3.2]{JonMan}.
Here, for $n=1$, we consider $\Delta(f^0-x_0)=\Delta(z-x_0)$ to be $1$, so that
\[ \Delta(f-x_0) = (-1)^{d (d-1)/2} d^{d} A^{d-1} \prod_{f'(\gamma)=0} \big( f(\gamma)-x_0 \big) .\]

When $f$ is a cubic polynomial with lead coefficient $A\in K^{\times}$
and critical points $\gamma_1,\gamma_2\in\Kbar$, 
equation~\eqref{eq:polyiter} becomes
\begin{equation}
\label{eq:cubiciter}
\Delta\big(f^n- x_0 \big)
= - 3^{3^n} A^{3^{2n-1}-1}\big( \Delta(f^{n-1}-x_0)\big)^3
\cdot \big( f^n(\gamma_1)-x_0 \big) \big( f^n(\gamma_2)-x_0 \big).
\end{equation}

\subsection{Sign relations for tree automorphisms}
\label{ssec:partree}
For the rest of Section~\ref{sec:collideQ}, consider a cubic polynomial $f\in K[z]$
as in Theorem~\ref{thm:main1}.
The two critical points $\gamma_1,\gamma_2\in\Kbar$ of $f$
are roots of the quadratic polynomial $f'(z)$,
so either they are both $K$-rational, or else they are Galois conjugate over $K$.
(The field $K(\gamma_1)=K(\gamma_2)$ is separable and hence Galois over $K$, since $\charact K\neq 2$.)
Either way, the common value $f^{\ell}(\gamma_1)=f^{\ell}(\gamma_2)$ is fixed by $\Gal(K(\gamma_1)/K)$
and hence is $K$-rational, proving the first claim of Theorem~\ref{thm:main1}.



Also for the rest of Section~\ref{sec:collideQ}, we will identify the backward orbit $\Orb_f^-(x_0)$
with the ternary rooted tree $T_{3,\infty}$. Given a labeling of this tree,
it will be convenient to define
\[ S(\sigma,y):= \sgn_{\ell}(\sigma,y) \sgn_{\ell-1}(\sigma,y) \frac{\sigma(\sqrt{-3})}{\sqrt{-3}} \in \{\pm 1\} \]
for any node $y$ of the tree and any $\sigma\in G_{K,\infty}$.

\begin{lemma}
\label{lem:sgncond}
Fix any labeling of the tree $\Orb_f^{-}(x_0)\cong T_{3,\infty}$. Let $y\in \Orb_f^{-}(x_0)$,
and let $\sigma,\sigma'\in G_{K,\infty}$.
\begin{enumerate}
\item If $\sigma(y)=y$, then
$\dsps S(\sigma,y)=+1$.
\item $S(\sigma' \sigma,y) = S(\sigma',\sigma(y)) \cdot S(\sigma,y)$.
\item $S(\sigma^{-1},\sigma(y)) = S(\sigma,y)$.
\item If $\sigma(y)=\sigma'(y)$, then
$\dsps S(\sigma,y) = S(\sigma',y)$.
\end{enumerate}
\end{lemma}

\begin{proof}
\textbf{Statement (1)}:
By equation~\eqref{eq:cubiciter} with $y$ in place of $x_0$, we have
\[ \Delta(f^{\ell}-y)\Delta(f^{\ell-1}-y) = - 3^{3^\ell}A^{3^{2\ell-1}-1} \big(\Delta(f^{\ell-1}-y)\big)^4
\big(f^{\ell}(\gamma_1) - y\big)\big(f^{\ell}(\gamma_2) - y\big) =-3 D_\ell(y)^2, \]
where
\begin{equation}
\label{eq:Dellydef}
D_{\ell}(y):=3^{(3^\ell - 1)/2} A^{(3^{2\ell-1}-1)/2} \big(\Delta(f^{\ell-1}-y)\big)^2 \big(f^{\ell}(\gamma_1) - y\big)\in K(y),
\end{equation}
since $f^{\ell}(\gamma_1) = f^{\ell}(\gamma_2)\in K$.

Because $\sigma$ fixes $y$, we have $\sigma\in\Gal(K_{\infty}/K(y))$, and it permutes the roots
both of $f^{\ell}-y$ and of $f^{\ell-1}-y$. The signs of those two permutations are
\[ \sgn_{\ell}(\sigma,y) = \frac{\sigma(\sqrt{\Delta(f^{\ell}-y)})}{\sqrt{\Delta(f^{\ell}-\sigma(y))}}
\quad \text{and} \quad
\sgn_{\ell-1}(\sigma,y) = \frac{\sigma(\sqrt{\Delta(f^{\ell-1}-y)})}{\sqrt{\Delta(f^{\ell-1}-\sigma(y))}}.\]
(Here, we define $\sqrt{\Delta(f^n-w)}$ as the appropriate power of
the lead coefficient multiplied by $\prod_{i<j} (\alpha_i-\alpha_j)$, where $\alpha_i$ ranges through the
roots of $f^n-w$, with the ordering determined by the labeling of the nodes of the tree.)
Thus,
\[ \sgn_{\ell}(\sigma,y) \sgn_{\ell-1}(\sigma,y)
= \frac{\sigma\big(\sqrt{-3}\cdot D_{\ell}(y)\big)}{\sqrt{-3}\cdot D_{\ell}\big(\sigma(y)\big)}
= \frac{\sigma\big(\sqrt{-3}\big)\cdot D_{\ell}\big(\sigma(y)\big)}{\sqrt{-3}\cdot D_{\ell}\big(\sigma(y)\big)}
= \frac{\sigma(\sqrt{-3})}{\sqrt{-3}}.\]
It follows that $S(\sigma,y)=+1$.

\textbf{Statement (2)}: This is immediate from equation~\eqref{eq:Parident}
and the definition of $S$, together with the observation that
\[ \frac{\sigma' \sigma(\sqrt{-3})}{\sqrt{-3}} = 
\frac{\sigma'\big(\sigma(\sqrt{-3})\big)}{\sigma(\sqrt{-3})} \cdot \frac{\sigma(\sqrt{-3})}{\sqrt{-3}} = 
\frac{\sigma'(\sqrt{-3})}{\sqrt{-3}} \cdot \frac{\sigma(\sqrt{-3})}{\sqrt{-3}} .\]

\textbf{Statement (3)}: This is immediate from Statement~(2) with $\sigma'=\sigma^{-1}$,
together with the observation that $S(e,y)=+1$, where $e\in G_{K,\infty}$ is the identity element.

\textbf{Statement (4)}: Because $\sigma(y)=\sigma'(y)$, we have
\[ S(\sigma,y) S(\sigma',y) = S\big(\sigma^{-1},\sigma(y)\big) S(\sigma',y) =
S\big(\sigma^{-1},\sigma'(y)\big) S(\sigma',y) =  S(\sigma^{-1} \sigma',y) = +1,\]
where the first equality is by Statement~(3), the third is by Statement~(2), and the fourth is by Statement~(1).
It follows that $S(\sigma,y) = S(\sigma',y)$.
\end{proof}

\subsection{Proving Theorem~\ref{thm:main1}}
\label{ssec:pfthm1}

We already proved the first claim of our first main theorem,
that $f^{\ell}(\gamma_i)$ is $K$-rational,
at the start of Section~\ref{ssec:partree}.
We now turn to the rest of the statement.

\begin{proof}[Proof of Theorem~\ref{thm:main1}]
\textbf{Step 1}.
Our main goal is to constuct a labeling of the tree $\Orb_f^{-}(x_0) \cong T_{3,\infty}$
for which $S(\sigma,y)=+1$ for each $y\in\Orb_f^{-}(x_0)$
and each $\sigma\in G_{K,\infty}$.
To do so, first choose any labeling of the tree.
We will proceed inductively up the tree, making changes to this labeling as we go,
in order to realize the desired embedding of $G_{K,\infty}$ in $\widetilde{Q}_{\ell,\infty}$.

For any integer $m\geq 0$, suppose that we have already ensured that for all nodes $x$
up to level $m-1$, we have $S(\sigma,x)=+1$ for all $\sigma\in G_{K,\infty}$.
Given any node $y$ at level $m$, let $G_{K,\infty}(y)$ denote its Galois orbit, and 
for each $w\in G_{K,\infty}(y)$, choose $\sigma_w\in G_{K,\infty}$ such that $\sigma_w(y)=w$.
Define
\[ W_y :=  \{w\in G_{K,\infty}(y) \, | \, S(\sigma_w,y) = -1 \},\]
and observe that $y\not\in W_y$, by Lemma~\ref{lem:sgncond}(1).

We modify our labeling as follows:
for each $w\in W_y$, pick a node $w'$ lying $\ell-1$ levels above $w$,
and transpose the labels of two of the nodes lying one level above $w'$.
That is, we make a single transposition of labels $\ell$ levels above $w$,
which reverses the sign of $\sgn_{\ell}(\sigma_w,y)$ but does not
affect $\sgn_{i}(\tau,x)$ for any $\tau\in G_{K,\infty}$ and any node $x$ at any level $j$ with $i+j<m+\ell$.
As a result, we now have $S(\sigma_w,y)=+1$ for this new labeling, but we have
not changed $S(\tau,x)$ for any $\tau\in G_{K,\infty}$ and any node $x$ at level $m-1$ or below,
nor for any node $x$ at level $m$ outside the orbit $G_{K,\infty}(y)$.

After having made such a transposition of labels for each $w\in W_y$,
we claim that for any node $x\in G_{K,\infty}(y)$ and any $\tau\in G_{K,\infty}$
we have $S(\tau,x)=+1$. Indeed, setting $z:=\tau(x)\in G_{K,\infty}(y)$,
Lemma~\ref{lem:sgncond}(4) applied to $\tau$ and $\sigma_z \sigma_x^{-1}$ yields
\[ S(\tau,x) = S(\sigma_z \sigma_x^{-1},x)
= S(\sigma_z,y) S(\sigma_x^{-1},x) = S(\sigma_z,y) S(\sigma_x,y) = (+1)(+1)=+1, \]
where the second equality is by Lemma~\ref{lem:sgncond}(2), the third is by Lemma~\ref{lem:sgncond}(3),
and the fourth is by our adjustments above to the labeling.

After applying this same relabeling process to each of the finitely many Galois orbits of nodes at level $m$,
we are left with a labeling for which $S(\sigma,y)=+1$ for every node $y$ at level $m$ of the tree,
while preserving the property that $S(\sigma,x)=+1$ for every node $x$ at lower levels of the tree.
This completes our induction, yielding the desired labeling of the full tree $\Orb_f^{-}(x_0)\cong T_{3,\infty}$.

\medskip

\textbf{Step 2}.
With this labeling now fixed, given any $\sigma\in G_{K,\infty}$ and any node $y$ of of the tree $\Orb_f^{-}(x_0)$,
we have
\[ \frac{\sgn_{\ell}(\sigma,y)\sgn_{\ell-1}(\sigma,y)}{\sgn_{\ell}(\sigma,x_0)\sgn_{\ell-1}(\sigma,x_0)}
= \frac{S(\sigma,y)}{S(\sigma,x_0)} = +1, \]
and hence $\sigma\in \widetilde{Q}_{\ell,\infty}$. Thus, $G_{K,\infty}\subseteq\widetilde{Q}_{\ell,\infty}$.

If $\sqrt{-3}\in K$, then
for any $\sigma\in G_{K,\infty}$ we have $\sigma(\sqrt{-3})=\sqrt{-3}$, and hence
\[ \sgn_{\ell}(\sigma,x_0)\sgn_{\ell-1}(\sigma,x_0)
= \sgn_{\ell}(\sigma,x_0)\sgn_{\ell-1}(\sigma,x_0) \frac{\sigma(\sqrt{-3})}{\sqrt{-3}}
= S(\sigma,x_0)=+1, \]
proving that $\sigma\in Q_{\ell,\infty}$, and hence that $G_{K,\infty}\subseteq Q_{\ell,\infty}$.

Conversely, if $\sqrt{-3}\not\in K$, then because
\[ \sqrt{-3} = D_\ell(x_0)^{-1} \sqrt{\Delta(f^{\ell}-x_0)\Delta(f^{\ell-1}-x_0)} \in K_{\ell} \subseteq K_{\infty} \]
where $D_{\ell}$ is as in equation~\eqref{eq:Dellydef},
there is some $\tau\in G_{K,\infty}$ such that $\sigma(\sqrt{-3})=-\sqrt{-3}$.
Thus, even if we were to choose a different labeling of the tree, we have
\[ \sgn_{\ell}(\sigma,x_0)\sgn_{\ell-1}(\sigma,x_0) = - S(\sigma,x_0) = -1, \]
where the last equality is by Lemma~\ref{lem:sgncond}(1).
Therefore, $G_{K,\infty}$ cannot be isomorphic to a subgroup of $Q_{\ell,\infty}$.
\end{proof}

\section{Generating large subgroups of tree automorphisms}
\label{sec:surj}
To prove Theorem~\ref{thm:main2}, we will need some purely 
group-theoretic results, in order to show that certain sets of
tree automorphisms generate the relevant groups $Q_{\ell,\infty}$ and $\widetilde{Q}_{\ell,\infty}$.

\subsection{Special properties of tree automorphism groups}
\label{ssec:surjpre}
To this end, for integers $\ell\geq 2$ and $n\geq 1$, we define the finite groups
\[ Q_{\ell,n} \subseteq \widetilde{Q}_{\ell,n}\subseteq\Aut(T_{3,n}) \]
to be the quotients of $Q_{\ell,\infty}$ and $\widetilde{Q}_{\ell,\infty}$
formed by restricting automorphisms of $T_{3,\infty}$ to the finite subtree $T_{3,n}$.

Note that for $n\geq \ell+1$, the groups
$Q_{\ell,n}$ and $\widetilde{Q}_{\ell,n}$ depend on the choice of labeling of the tree.
However, for $n=\ell$, the groups $Q_{\ell,\ell}$ and $\widetilde{Q}_{\ell,\ell}$
are defined independent of the labeling, since the only node $y$ of $T_{3,\ell}$ for
which $\sgn_\ell(\cdot,y)$ is defined at all is the root node $y=x_0$.
More precisely, $\widetilde{Q}_{\ell,\ell}=\Aut(T_{3,\ell})$, and
$Q_{\ell,\ell}$ is the index-2 subgroup consisting of $\sigma\in\Aut(T_{3,\ell})$
for which $\sgn_{\ell-1}(\sigma,x_0) = \sgn_{\ell}(\sigma,x_0)$.
Similarly, for $n\leq \ell+1$, we have
$Q_{\ell,n} = \widetilde{Q}_{\ell,n} =\Aut(T_{3,n})$.

We will need to define certain subgroups of $Q_{\ell,\ell}$ that \emph{do} depend on the labeling.
In particular, if we fix a labeling of $T_{3,\ell}$, and if we denote by $x_{00},x_{01},x_{02}$
the three nodes of the tree at level $1$, then the set
\begin{equation}
\label{eq:Hdef}
H:= \big\{ \sigma\in Q_{\ell,\ell} \, \big| \,
\sgn_{\ell-1}(\sigma,x_{00}) = \sgn_{\ell-1}(\sigma,x_{01}) = \sgn_{\ell-1}(\sigma,x_{02}) \big\}
\end{equation}
is clearly a subgroup of $Q_{\ell,\ell}$.
Because of the two sign conditions, the index of $H$ in $Q_{\ell,\ell}$ is $[Q_{\ell,\ell}:H]=4$.

Less obviously, $H$ is a non-normal subgroup of $Q_{\ell,\ell}$; in fact, there are four conjugate
subgroups (including $H$ itself) which we will arbitrarily denote $H_{\ell,1},H_{\ell,2},H_{\ell,3},H_{\ell,4}$.
To see this, first observe that changing the labels of nodes at level $\ell-1$ or below has no effect
on $\sgn_{\ell-1}(\sigma,x_{0i})$ for any $i$, and hence such a change leaves $H$ the same.
The same is true if we change the labeling by making an even permutation of labels of the nodes
at level $n$ above any one (or more) of $x_{00},x_{01},x_{02}$.
On the other hand, if we make an odd permutation of the labels at level $n$ above $x_{0i}$,
then for any $\sigma\in Q_{\ell,\ell}$ for which $\sigma(x_{0i})\neq x_{0i}$,
the effect is to multiply both $\sgn_{\ell-1}(\sigma,x_{0i})$ and $\sgn_{\ell-1}(\sigma,x_{0j})$ by $-1$,
where $\sigma^{-1}(x_{0i})=x_{0j}$; since the sign of $\sigma$ above the third node among $x_{00},x_{01},x_{02}$
remains unchanged, this means that the subgroup described by condition~\eqref{eq:Hdef} changes.
A similar effect occurs if we make odd-parity label changes above two of $x_{00},x_{01},x_{02}$.
However, if we make odd-parity label changes above all three, then the subgroup $H$ is preserved.
Thus, with two possible parity choices above each node $x_{00},x_{01},x_{02}$, modulo changing all three,
there are indeed $2^3/2=4$ possible (conjugate) subgroups $H$ as in equation~\eqref{eq:Hdef}.

For any two nodes $a,b$ at the same level $m$ of a tree $T_{d,n}$ or $T_{d,\infty}$, we may define
their \emph{tree distance} $\dist_T(a,b)$ to be the smallest integer $j\geq 0$ for which $a$ and $b$ both
lie above a common node $c$ at level $m-j$.
For example, then, two different nodes sharing a common parent are distance 1 apart,
whereas those sharing a common grandparent but not a parent are distance 2 apart. Equivalently,
$j=\dist_T(a,b)$ is half of the graph distance $2j$ between $a$ and $b$, since getting from $a$ to $b$
on the graph requires going down $j$ levels from $a$ to $c$, and then up $j$ levels up to $b$.

\begin{defin}
\label{def:arb2trans}
Let $d\geq 2$ and $n\geq 1$ be integers, and let $G\subseteq\Aut(T_{d,n})$ be a subgroup.
We say that $G$ is \emph{arboreally doubly transitive} at level $n$ if
for any nodes $a_1,a_2,b_1,b_2$ at level $n$ of the tree $T_{d,n}$ for which
$\dist_T(a_1,b_1)=\dist_T(a_2,b_2)$, 
there is some $\sigma\in G$ such that $\sigma(a_1)=a_2$ and $\sigma(b_1)=b_2$.
\end{defin}

It is an easy induction to show that $\Aut(T_{d,n})$ is arboreally doubly transitive
at every level,
as are the subgroups $Q_{\ell,n}\subseteq\widetilde{Q}_{\ell,n}\subseteq\Aut(T_{3,n})$.

\subsection{Group-theoretic results on the tree}
\label{ssec:surjlem}

\begin{lemma}
\label{lem:evengen}
Let $n\geq m \geq 2$ be integers.
Let $G\subseteq \Aut(T_{3,n})$ be a subgroup with the following properties.
\begin{itemize}
\item The quotient of $G$ formed by restricting to the subtree $T_{3,n-1}$ rooted at $x_0$
is arboreally doubly transitive at level $n-1$.
\item  $G$ acts transitively on the nodes of $T_{3,n}$ at level $n$.
\item  There exists $\theta\in G$ that fixes almost every node of the tree $T_{3,n}$, with the following exceptions.
There are nodes $z_0$ and $z_1$ at level $n-1$ with $\dist_T(z_0,z_1)=m-1$,
and $\theta$ acts as a transposition of two pairs of nodes at level $n$ of the tree,
with one pair above $z_0$ and the other above $z_1$.
\end{itemize}
Then for any $\sigma\in\Aut(T_{3,n})$ that fixes every node of the tree at level $n-1$ and below,
and for which $\sgn_m(\sigma,x)=+1$ for every node $x$ at level $n-m$,
we have $\sigma\in G$.
\end{lemma}

\begin{proof}
\textbf{Step 1}.
For any two nodes $a$ and $b$ at level $n-1$ of $T_{3,n}$ with $1\leq \dist_T(a,b)\leq m-1$,
we claim that there is some $\rho_{ab}\in G$
that fixes almost every node of the tree,
except that it acts as a transposition of two nodes at level $n$ above $a$,
and also as a transposition of two nodes at level $n$ above $b$.
To see this, let $w$ be the node at level $n-m$ that $a,b$ both lie above,
let $w_0,w_1,w_2$ be the three nodes at level $n-m+1$ immediately above $w$,
and let $c$ be a node at level $n-1$ not lying above the same one of $w_0,w_1,w_2$,
as either $a$ or $b$ does.
Choose $\tilde{\tau}_a,\tilde{\tau}_b$ in the restriction of $G$ to $T_{3,n-1}$ such that
\[ \tilde{\tau}_a(z_0)=a, \quad 
\tilde{\tau}_a(z_1)=c, \quad
\tilde{\tau}_b(z_0)=b, \quad
\tilde{\tau}_b(z_1)=c.\]
Such $\tilde{\tau}_a$ exists by the arboreal double transitivity hypothesis,
and because $\dist_T(a,c)=m-1=\dist_T(z_0,z_1)$; similarly for $\tilde{\tau}_b$.
Thus, there exist
$\tau_a,\tau_b\in G$ that restrict to $\tilde{\tau}_a,\tilde{\tau}_b$, respectively.
Then using the element $\theta\in G$ of the third bullet point,
the composition $\theta_a:=\tau_a\theta\tau_a^{-1}$ acts as transpositions above $a$ and $c$,
and $\theta_b:=\tau_b\theta\tau_b^{-1}$ acts as transpositions above $b$ and $c$,
fixing every other node. Hence, $\theta_a \theta_b$ acts as a transposition above both $a$ and $b$,
and as either the identity or a $3$-cycle above $c$. Therefore, $\rho_{ab}:=(\theta_a \theta_b)^3\in G$
has the desired property, proving our first claim.

\medskip
\textbf{Step 2}.
Next, we claim that for any node $a$ at level $n-1$, there is some $\mu_a\in G$
that acts as a 3-cycle above $a$ and fixes every other node of $T_{3,n}$.
To prove this claim,
let $b,c$ be the two other nodes at level $n-1$ with $\dist_T(a,b)=\dist_T(a,c)=1$,
and consider the elements $\rho_{ab},\rho_{ac}\in G$ from Step~1.

If the two nodes $a_0,a_1$ above $a$ transposed by $\rho_{ab}$ coincide with
the two transposed by $\rho_{ac}$, then let $\lambda\in G$ map the third node $a_2$ above $a$ to $a_0$;
such $\lambda$ exists by the transitivity hypothesis. Otherwise, let $\lambda$ be the identity element of $G$.

Thus, $\rho_{ab}\in G$ and $\lambda\rho_{ac}\lambda^{-1}\in G$ act as different transpositions above $a$,
so that their product $\rho_{ab}\lambda\rho_{ac}\lambda^{-1}$ acts as a 3-cycle above $a$, and as
a transposition above each of $b$ and $c$. Hence, $\mu_a:=(\rho_{ab}\lambda\rho_{ac}\lambda^{-1})^2\in G$
is a 3-cycle above $a$, while fixing every other node of the tree. This proves our second claim.

\medskip
\textbf{Step 3}.
Consider any $\sigma\in\Aut(T_{3,n})$ 
that fixes every node of the tree at level $n-1$ and below
and for which $\sgn_m(\sigma,x)=+1$ for every node $x$ at level $n-m$.
Then $\sigma$ may be written as a product of disjoint $3$-cycles and pairs
of disjoint $2$-cycles, where any such pairs lie above nodes $a,b$ at level $n-1$ with $\dist_T(a,b)\leq m-1$.

Let $\sigma'$ be the product of $\sigma$ with the corresponding $\rho_{ab}$ from Step~1 for any such
pair $a,b$. Then $\sigma'$ is a product of disjoint 3-cycles at level $n$.
(Note that the 2-cycle above $a$ or $b$
in any such $\rho_{ab}$ may not exactly match the corresponding $2$-cycle in $\sigma$, but in that case
their product is a 3-cycle.) Each such $3$-cycle is of the form $\mu_a$ or $\mu_a^2$ for some $\mu_a$
from Step~2. Thus, $\sigma'$ and hence $\sigma$ is a product of $\rho_{ab}$'s and $\mu_a$'s and therefore
belongs to $G$.
\end{proof}

\begin{thm}
\label{thm:gen}
Let $n\geq \ell\geq 2$ be integers. Fix a labeling on the tree $T_{3,n}$,
and let $G\subseteq Q_{\ell,n}$ be a subgroup with the following properties.
\begin{itemize}
\item
The quotient of $G$ formed by restricting to
the subtree $T_{3,n-1}$ rooted at $x_0$ is all of $Q_{\ell,n-1}$.
\item
Let $y$ be one of the nodes at level $1$, and consider the subgroup $G^{y}$
of elements of $G$ that fix $y$. Then the quotient of $G^y$ formed by restricting to the subtree $T_{3,n-1}$
rooted at $y$ is all of $Q_{\ell,n-1}$.
\item
If $n=\ell$, then the set of elements of $G$ that fix the bottom $\ell-1$ rows of $T_{3,\ell}$ is
not contained in any of the subgroups $H_{\ell,1},H_{\ell,2},H_{\ell,3},H_{\ell,4}$.
\end{itemize}
Then $G=Q_{\ell,n}$.
\end{thm}

\begin{proof}
\textbf{Step 1}.
We claim that $G$ acts transitively at level $n$ of the tree $T_{3,n}$.
To see this, given nodes $a,b$ at level $n$, by the first bulleted hypothesis there exist $\tau_a,\tau_b\in G$
such that $\tau_a(a)$ and $\tau_b(b)$ both lie above the node $y$.
Then by the second bulleted hypothesis, there is some $\sigma\in G$ such that $\sigma(\tau_a(a))=\tau_b(b)$.
Hence, the composition $\tau_b^{-1}\sigma\tau_a$ maps $a$ to $b$, proving our claim.

If $n=2$, in which case $\ell=2$ also,
then jump ahead to Step~5 below.

\medskip

\textbf{Step 2}.
For Steps~2--4, we assume that $n\geq 3$.
Define
\[ m:= \begin{cases}
\ell & \text{ if } n\geq \ell+1, \\
\ell-1 & \text{ if } n=\ell.
\end{cases} \]
Let $c$ be a node at level $n-m\geq 1$ of the tree, with $c$ lying on or above the node $y$
specified in the hypotheses.
Let $c_0,c_1,c_2$ be the three nodes at level $n-m+1\leq n-1$ immediately above $c$.
By the first bulleted hypothesis, there is some $\tau\in G$ that acts as a 3-cycle on
$\{c_{0},c_{1},c_{2}\}$ but otherwise acts as the identity on $T_{3,n-1}$.
That is, $\tau$ fixes all nodes of $T_{3,n-1}$
that do not lie at or above $c$; and
for any word $w$ of length at most $m-1$ in the symbols $\{0,1,2\}$, we have
$\tau(c_{0}w)=c_{1} w$, $\tau(c_{1} w)=c_{2} w$, and $\tau(c_{2} w)=c_{0} w$.
In particular, if we fix a node $z_0$ at level $n-1$ lying above $c_0$, then
$\tau$ also acts as a 3-cycle on $\{z_0,z_1,z_2\}$,
where $z_1:=\tau(z_0)$ and $z_2:=\tau(z_0)$, which are nodes at level $n-1$
lying above $c_{1}$ and $c_{2}$, respectively.
There is indeed such an element of $Q_{\ell,n-1}$ because $\tau$ acts on every node at levels up to $n-1$
either as part of a 3-cycle or by fixing it, and hence $\tau$ is even at every level of every subtree of $T_{3,n-1}$.
See Figure~\ref{fig:thmgen}.

By the second bulleted hypothesis, there is some $\sigma\in G$ that acts as a transposition of two
of the nodes above $z_0$ and two of the nodes above $z_2$,
but which fixes every other node lying on or above $y$.
Since $\dist_T(z_0,z_2)=m-1<\ell$, there is indeed an element of $Q_{\ell,n-1}$ acting in this
way on the subtree of $T_{3,n}$ rooted at $y$, so our hypotheses
do indeed show that such $\sigma\in G$ exists.
However, we do not know how $\sigma$ behaves outside of the subtree rooted at $y$,
just we do not know how $\tau$ acts on any of the nodes at level $n$.

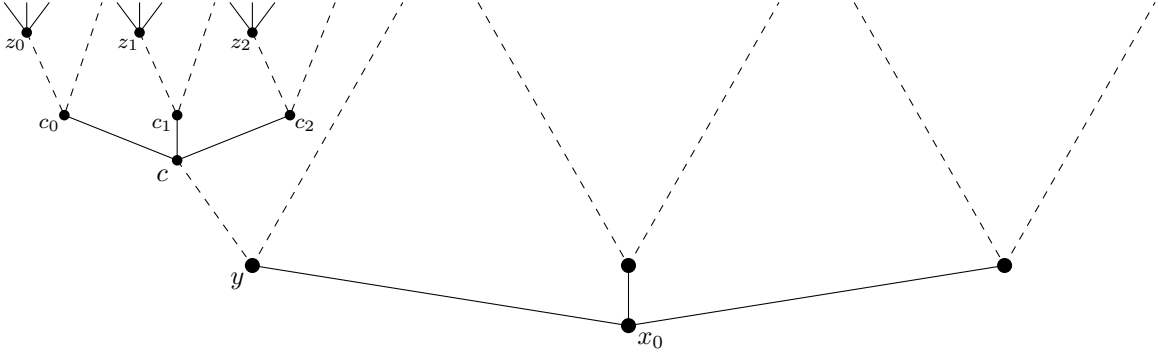
\begin{figure}
\begin{tikzpicture}
\path[draw] (0,0) -- (0,.8);
\path[draw] (-5,.8) -- (0,0) -- (5,.8);
\path[draw] (-7.5,2.8) -- (-6,2.2) -- (-4.5,2.8);
\path[draw] (-6,2.8) -- (-6,2.2);
\path[draw,dashed] (-6,2.2) -- (-5,.8) -- (-3,4.3);
\path[draw,dashed] (-2,4.3) -- (0,.8) -- (2,4.3);
\path[draw,dashed] (3,4.3) -- (5,.8) -- (7,4.3);
\path[draw,dashed] (-8,3.9) -- (-7.5,2.8) -- (-7.0,4.3);
\path[draw,dashed] (-6.5,3.9) -- (-6,2.8) -- (-5.5,4.3);
\path[draw,dashed] (-5,3.9) -- (-4.5,2.8) -- (-3.9,4.3);
\path[draw] (-8.3,4.3) -- (-8,3.9) -- (-7.7,4.3);
\path[draw] (-8,4.3) -- (-8,3.9);
\path[draw] (-6.8,4.3) -- (-6.5,3.9) -- (-6.2,4.3);
\path[draw] (-6.5,4.3) -- (-6.5,3.9);
\path[draw] (-5.3,4.3) -- (-5,3.9) -- (-4.7,4.3);
\path[draw] (-5,4.3) -- (-5,3.9);
\path[fill] (0,0) circle (0.1);
\path[fill] (0,.8) circle (0.1);
\path[fill] (-5,.8) circle (0.1);
\path[fill] (5,.8) circle (0.1);
\path[fill] (-6,2.2) circle (0.07);
\path[fill] (-7.5,2.8) circle (0.07);
\path[fill] (-6,2.8) circle (0.07);
\path[fill] (-4.5,2.8) circle (0.07);
\path[fill] (-8,3.9) circle (0.07);
\path[fill] (-6.5,3.9) circle (0.07);
\path[fill] (-5,3.9) circle (0.07);
\node (x0) at (0.3,-0.2) {\footnotesize $x_0$};
\node (y) at (-5.2,0.6) {\footnotesize $y$};
\node (x) at (-6.2,2) {\footnotesize $c$};
\node (y1) at (-7.7,2.67) {\tiny $c_0$};
\node (y2) at (-6.2,2.67) {\tiny $c_1$};
\node (y3) at (-4.3,2.67) {\tiny $c_2$};
\node (z1) at (-8.15,3.74) {\tiny $z_0$};
\node (z2) at (-6.65,3.74) {\tiny $z_1$};
\node (z3) at (-5.15,3.74) {\tiny $z_2$};
\end{tikzpicture}
\caption{The tree $T_{3,n}$ of Theorem~\ref{thm:gen}.
The node $c$ is at level $n-m$, and each node $z_i$ is at level $n-1$, above the corresponding node $c_i$
at level $n-m+1$. The automorphism $\tau$ is a 3-cycle on $\{c_0,c_1,c_2\}$ and also on $\{z_0,z_1,z_2\}$,
with $\sigma$ transposing two of the nodes above $z_0$, and also transposing two of the nodes above $z_2$.}
\label{fig:thmgen}
\end{figure}

\medskip

\textbf{Step 3}.
Define $\lambda\in G$ to be the commutator $\lambda:=\sigma\tau\sigma^{-1}\tau^{-1}$,
where $\sigma,\tau\in G$ are the specific elements from Step~2.
We claim that
\begin{itemize}
\item $\lambda(x')=x'$ for every node $x'$ at levels $0$ through $n-1$,
\item $\sgn_1(\lambda,z_1)=-1$ and $\sgn_1(\lambda,z_2)=-1$, and
\item $\sgn_1(\lambda,z)=+1$ for every other node $z$ at level $n-1$ lying above $y$.
\end{itemize}
(However, we make no claim about $\sgn_1(\lambda,z)$ for nodes $z$ at level $n-1$ not lying above $y$.)

To prove the first point of the claim, given a node $x'$ at level $n-1$ or below, there are two cases.
On the one hand, if $x'$ lies above $y$, then so does $\tau^{-1}(x')$, and hence $\sigma$ fixes
both $x'$ and $\tau^{-1}(x')$. Therefore $\lambda(x')=\sigma\tau\tau^{-1}(x')=x'$.
On the other hand, if $x'$ does not lie above $y$, then neither does $\sigma^{-1}(x')$, and hence $\tau$
fixes both $x'$ and $\sigma^{-1}(x')$. It follows that $\lambda(x')=\sigma\tau\sigma^{-1}(x')=x'$,
proving the first bullet point of our claim.

For the second and third bullet points of our claim,
given an arbitrary node $z$ at level $n-1$ that lies above $y$,
we have $\sigma(z)=z$.
In addition, $\sgn_1(\sigma,z)=+1$ unless $z=z_0$ or $z=z_2$, in which case $\sgn_1(\sigma,z)=-1$.
By equation~\eqref{eq:Parident}, we have
\begin{align*}
\sgn_1(\lambda,z) &= 
\sgn_1(\sigma,z) \sgn_1\big(\tau,\tau^{-1}(z)\big)
\sgn_1\big(\sigma^{-1},\tau^{-1}(z)\big) \sgn_1(\tau^{-1},z)
\\
&=\sgn_1(\sigma,z) \sgn_1(\sigma^{-1},\tau^{-1}(z)) \big( \sgn_1(\tau^{-1},z) \big)^2
=\sgn_1(\sigma,z) \sgn_1(\sigma,\tau^{-1}(z)).
\end{align*}
Thus, recalling that $\tau^{-1}$ maps $z_0$ to $z_2$ to $z_1$ to $z_0$, we have
\begin{equation}
\label{eq:lambdasign}
\sgn_1(\lambda,z) = \begin{cases}
(-1)(-1) = +1 & \text{ if } z=z_0, \\
(+1)(-1) = -1 & \text{ if } z=z_1, \\
(-1)(+1) = -1 & \text{ if } z=z_2, \\
(+1)(+1) = +1 & \text{ otherwise},
\end{cases}
\end{equation}
completing the proof of the claim.

\medskip

\textbf{Step 4}.
Define $\mu\in G$ to be the commutator $\mu:=\lambda\tau\lambda^{-1}\tau^{-1}$,
where $\tau,\lambda\in G$ are the specific elements from Steps~2 and~3.
Then $\mu$ fixes every node $x'$ at levels $0$ through $n-1$; this is because $\lambda$ does,
and hence $\mu(x') = \tau(\tau^{1}(x'))=x'$.
In addition, by a similar computation as in equation~\eqref{eq:lambdasign},
for any node $z$ at level $n-1$ that lies above $y$, we have
\[ \sgn_1(\mu,z) = \begin{cases}
(+1)(-1) = -1 & \text{ if } z=z_0, \\
(-1)(+1) = -1 & \text{ if } z=z_1, \\
(-1)(-1) = +1 & \text{ if } z=z_2, \\
(+1)(+1) = +1 & \text{ otherwise}.
\end{cases} \]
For the other nodes $z$ at level $n-1$, this time we \emph{can} compute $\sgn_1(\mu,z)$,
since $\lambda(z)=\tau(z)=z$ for such $z$, and hence
\begin{align*}
\sgn_1(\mu,z) &= 
\sgn_1(\lambda,z) \sgn_1(\tau,z) \sgn_1(\lambda^{-1},z) \sgn_1(\tau^{-1}z) 
\\
&= \big( \sgn_1(\lambda,z) \big)^2 \big( \sgn_1(\tau,z) \big)^2 = +1.
\end{align*}
Thus, at level $n-1$ of $T_{3,n}$, $\mu$ acts as a $2$-cycle above $z_0$ and $z_1$,
and as either the identity or a $3$-cycle above every other node $z$ at level $n-1$.
It follows that $\theta:=\mu^3\in G$ is the identity on the whole tree $T_{3,n}$ except for
a $2$-cycle above $z_0$ and another $2$-cycle above $z_1$.

By Lemma~\ref{lem:evengen}, 
$G$ contains every $\sigma\in\Aut(T_{3,n})$ that fixes every node of the tree at level $n-1$ and below,
and for which $\sgn_m(\sigma,x)=+1$ for every node $x$ at level $n-m$.

\medskip

\textbf{Step 5}.
We claim that $G$ also contains the subgroup $W_{\ell,n}\subseteq Q_{\ell,n}$ consisting of
every $\sigma\in\Aut(T_{3,n})$ that fixes every node of the tree at level $n-1$ and below,
and for which $\sgn_\ell(\sigma,x)=+1$ for every node $x$ at level $n-\ell$.
If $n\geq\ell+1$, then the integer $m$ from Step~2 is $m=\ell$,
and hence we have just proven the claim at the end of Step~4.

To prove the claim, then, it remains to consider the case that $n=\ell\geq 2$.
In that case, define $x_{00},x_{01},x_{02}$ to be the three nodes at level~1 of the tree.
The third bulleted hypothesis says that $G\subseteq Q_{\ell,\ell}$ contains an element $\rho$
that is the identity on the subtree $T_{3,n-1}$ rooted at $x_0$ and for which two of
$\sgn_{\ell-1}(\rho,x_{00})$, $\sgn_{\ell-1}(\rho,x_{01})$, and $\sgn_{\ell-1}(\rho,x_{02})$
are $-1$, and the third is $+1$. Without loss, the two negative signs occur over $x_{00}$ and $x_{01}$.

If $n=\ell=2$, then $\theta':=\rho^3\in G$ consists of two transpositions at level $2$,
one above each of $x_{00}$ and $x_{01}$.
Otherwise, i.e., if $n=\ell\geq 3$, then there is some $\sigma\in\Aut(T_{3,n})$
that is also the identity on $T_{3,n-1}$
but with $\sgn_{\ell-1}(\sigma,x_{0i})=+1$ for each $i=0,1,2$,
and so that $\theta':=\rho\sigma\in G$ consists of only two transpositions at level $n=\ell$,
one above each of $x_{00}$ and $x_{01}$.
Since $\sigma\in G$ by the conclusion of Step~4,
we also have $\theta'\in G$. Either way, then, by Lemma~\ref{lem:evengen} with $\ell$ in place of $m$,
our claim follows for all cases, including $n=\ell\geq 2$.

Finally, let $\phi:Q_{\ell,n}\to Q_{\ell,n-1}$ be the surjective homomorphism
given by restricting to the first $n-1$ levels of the tree. Then by definition of $Q_{\ell,n}$,
we have $\ker\phi=W_{\ell,n}$. By the above claim, therefore, we have $\ker\phi\subseteq G$.
In addition, the restriction of $\phi$ to $G$ also has image $Q_{\ell,n-1}$.
Hence, we must have $G=Q_{\ell,n}$.
\end{proof}

\section{Some technical results}
\label{sec:tech}

\subsection{A useful normal form}
The following result, a portion of which appeared as \cite[Proposition~4.2]{BDJKRZ},
shows that we may restrict our attention to cubic polynomials of the form $f(z)=Az^3+Bz+1$.

\begin{prop}
\label{prop:collide}
Let $K$ be a field of characteristic not equal to $3$.
Then $f$ is conjugate over $K$
to a cubic polynomial of the form either $Az^3+Bz+1$ or $Az^3+Bz$, but not both.
Moreover,
\begin{enumerate}
\item
If $f$ is conjugate to $Az^3+Bz+1$, then $A$ and $B$ are unique.
\item
If $\charact K\neq 2,3$ and the two finite critical points of $f$ collide, and if the forward orbit of these critical points
does not include a fixed point, then $f$ is conjugate to a unique polynomial of the form $Az^3+Bz+1$.
\end{enumerate}
\end{prop}

\begin{proof}
Write $f(z)=a_3 z^3+a_2 z^2+a_1 z+a_0$. Conjugating by the $K$-rational translation $z\mapsto z-a_2/(3a_3)$,
we may assume without loss that $a_2=0$, so that $f(z)=a_3 z^3+b_1 z+b_0$.
If $b_0=0$, then we already have the form $Az^3+Bz$. Otherwise, the $K$-rational scaling $z\mapsto b_0 z$
yields $b_0^{-1}f(b_0z)$ of the form $Az^3+Bz+1$.

If $f_1(z)=A_1 z^3 + B_1 z + C_1$ is conjugate to $f_2(z)=A_2 z^3 + B_2 z + C_2$ with $C_1,C_2\in\{0,1\}$,
then the conjugation
must fix the totally invariant point at $\infty$, so it must be of the form $z\mapsto \alpha z + \beta$.
Because both $f_1$ and $f_2$ lack a $z^2$ term and $\charact K\neq 3$, we must have $\beta=0$.
It follows that $C_2=\alpha C_1$. Since $C_1,C_2\in\{0,1\}$ and $\alpha\neq 0$, we must have $C_1=C_2$;
moreover, if $C_1=C_2=1$, then $\alpha=1$ and hence $f_1=f_2$.

Having proven the main claim and also Statement~(1), assume the hypotheses of Statement~(2).
After conjugating as above, we may assume $f(z)=Az^3+Bz+C$ with $C\in\{0,1\}$.
Thus, $f'(z)=3Az^2+B$,
so that the two critical points are $\pm\gamma$, where $\gamma^2=-B/(3A)$.

If $C=0$, then we have $f(-z)=-f(z)$, and hence $f^n(-\gamma)=-f^n(\gamma)$ for all $n\geq 1$.
Since $f^\ell(\gamma)=f^\ell(-\gamma)$ and $\charact K \neq 2$,
it follows that this common value is $0$, which is fixed by $f$,
contradicting the hypotheses. Thus, we must have $C=1$; and by Statement~(1),
the form $Az^3+Bz+1$ is unique.
\end{proof}

\begin{remark}
In the case that $f$ is conjugate to a polynomial of the form $Az^3+Bz$, then the value of $A$ is not unique.
Indeed, in that case $f$ is conjugate over $\Kbar$ to $A' z^3 + Bz$ for any $A'\in K^{\times}$,
via $z\mapsto \lambda z$ with $\lambda=\sqrt{A'/A}$;
and it is conjugate over $K$ if and only if $A'/A$ is a square in $K$.
\end{remark}

\begin{remark}
Proposition~\ref{prop:collide} includes the hypothesis that the critical points are not pre-fixed.
This hypothesis does not present a problem towards Theorem~\ref{thm:main2}, however.
After all, if the orbit of the colliding critical points is preperiodic, then the iterated discriminants
$\Delta(f^n-x_0)$ will eventually start repeating (up to square multiples). This situation
would make it impossible to have a sequence of pairwise distinct places $v_1,v_2,\ldots$
as in the hypotheses of Theorem~\ref{thm:main2}.
\end{remark}


\subsection{Some explicit computations}
\label{ssec:explicit}
For $f(z)=Az^3+Bz+1$,
we have $f'(z)=3Az^2+B$, and hence the critical points are $\pm\gamma\in\Kbar$,
where $\gamma^2=-B/(3A)$. Given the importance of the iterates of these critical points
in the discriminant formula~\eqref{eq:cubiciter}, we are motivated
to define sequences $F_n=F_n(A,B)$
and $G_n=G_n(A,B)$ by
\begin{equation}
\label{eq:FGdef}
f^n(\gamma)=F_n(A,B) \gamma + G_n(A,B) .
\end{equation}
That is, $F_0=1$ and $G_0=0$, and by substituting $f^{n-1}(\gamma)=F_{n-1}\gamma + G_{n-1}$
into $f(z)$ and recalling $\gamma^2 = -B/(3A)$, we have
\begin{equation}
\label{eq:Fiter}
F_n = \big( 3AG_{n-1}^2 +B - \frac{B}{3} F_{n-1}^2 \big) F_{n-1}
=  f'(G_{n-1}) F_{n-1} - \frac{B}{3} F_{n-1}^3
\end{equation}
and
\begin{equation}
\label{eq:Giter}
G_n = A G_{n-1}^3 + BG_{n-1} + 1 - B F_{n-1}^2 G_{n-1}
= f(G_{n-1}) - B F_{n-1}^2 G_{n-1} .
\end{equation}
In particular, note that $f^n(\gamma)=f^n(-\gamma)$ if and only if $F_n=0$.
That is, the two critical points $\pm\gamma$ of $f$ collide at the $\ell$-th iterate if and only if
$F_{\ell-1}\neq 0$ and $F_{\ell}=0$; or equivalently, if and only if
\begin{equation}
\label{eq:ellcond}
F_{\ell-1}\neq 0
\quad\text{and}\quad 9AG_{\ell-1}^2 +3B= B F_{\ell-1}^2 .
\end{equation}
For example, we have $F_1=2B/3$ and $G_1=1$, so
the two critical points collide at iterate $\ell=2$ if and only if $B\neq 0$ and $81A + 27B - 4B^3=0$.

\begin{remark}
\label{rem:BC}
When $f(z)=Az^3+Bz+1$, with $f'(z)=3Az^2+B$
and critical points $\gamma_1=\gamma$ and $\gamma_2=-\gamma$, we have
\[ f'\bigg(\frac{1}{2}(\gamma_1+\gamma_2)\bigg) = f'(0) = B, \]
which coincides with our definition of $B$ in equation~\eqref{eq:BCdef}.
In that same equation, for each $n\geq 1$, we have
\[ C_n = \big( (F_n \gamma_1 + G_n) - (F_n \gamma_2 + G_n) \big)^2 = 4F_n^2 \gamma^2
= -\frac{4B}{3A} F_n^2 .\]
We also note the following small redundancy in the hypotheses of Theorem~\ref{thm:main2}.
As noted above, we have $F_1=2B/3$, whence $C_1 = -16 B^3 / (27 A)$ in the notation
of equation~\eqref{eq:BCdef}.
Therefore the conditions $v_n(A)=v_n(B)=v_n(6)=0$ in Theorem~\ref{thm:main2}
already imply that $v_n(C_1)=0$.
However, $F_n$ and hence $C_n$ involve other terms for $n\geq 2$.
For example, the computation just before this remark shows that $F_2=(2B/81)(81A + 27B - 4B^3)$,
so that hypothesis~(2) of Theorem~\ref{thm:main2} is not redundant for $2 \leq j \leq \ell-1$.
\end{remark}

Further define $H_n=H_n(A,B)$ and a polynomial $E_n(t)=E_n(A,B,t)$ by
\begin{equation}
\label{eq:Endef}
H_n:= \frac{B}{3A} F_n^2 + G_n^2 \quad\text{and}\quad
E_n(t):=\big( f^n(\gamma)-t \big) \big( f^n(-\gamma)-t\big) .
\end{equation}
Then we have $E_n(t) = H_n - 2 G_n t + t^2$,
by equation~\eqref{eq:FGdef} and the fact that $\gamma^2 = -B/(3A)$.

\begin{prop}
\label{prop:alphacubic}
Let $f(z)=Az^3+Bz+1\in K[z]$ and $x_0\in K$, with $A\neq 0$.
Suppose that the two critical points $\pm\gamma\in\Kbar$ of $f$ collide at iterate $\ell\geq 2$.
Write $f^{-1}(x_0)=\{\alpha_1,\alpha_2,\alpha_3\}\subseteq\Kbar$. Then
\[ \prod_{i=1}^3 \big(z- E_{\ell-1}(\alpha_i)\big) =
z^3 - s_{\ell,1} z^2 + s_{\ell,2} z - s_{\ell,3} \]
where
\[ s_{\ell,1} := \frac{B}{A} + 12 G_{\ell-1}^2,
\quad
s_{\ell,2} := -\frac{6}{A} G_{\ell-1}\big( f^\ell(\gamma) - x_0 \big),
\quad
s_{\ell,3} := \frac{1}{A^2} \big( f^\ell(\gamma) - x_0 \big)^2 . \]
In addition, we have
\[ s_{\ell,2}^2 - 4s_{\ell,1}s_{\ell,3}
= -\frac{4B}{3A^3} F_{\ell-1}^2 \big( f^\ell(\gamma) - x_0 \big)^2 .\]
\end{prop}

\begin{proof}
This is a brute-force calculation. First, observe that equation~\eqref{eq:ellcond} yields
\begin{equation}
\label{eq:Hident}
H_{\ell-1} = \frac{B}{3A} F_{\ell-1}^2 + G_{\ell-1}^2
=  3G_{\ell-1}^2 + \frac{B}{A} + G_{\ell-1}^2 = \frac{B}{A} + 4 G_{\ell-1}^2.
\end{equation}
The elementary symmetric functions of $\alpha_1,\alpha_2,\alpha_3$ are
\[ \sigma_1:=\sigma_1(\{\alpha_i\}) = \alpha_1 + \alpha_2 + \alpha_3=0,
\quad
\sigma_2:=\sigma_2(\{\alpha_i\}) = \alpha_1 \alpha_2 + \alpha_1 \alpha_3 + \alpha_2 \alpha_3 = \frac{B}{A},\]
and
\[ \sigma_3:=\sigma_3(\{\alpha_i\}) = \alpha_1 \alpha_2 \alpha_3 = \frac{x_0-1}{A},\]
since $f(z)-x_0 = Az^3+Bz+1-x_0$ has roots $\{\alpha_i\}$.
Thus, the negative of the coefficient of $z^2$ in $q(z):=\prod_{i=1}^3 (z- E_{\ell-1}(\alpha_i))$
is $\sigma_1(\{E_{\ell-1}(\alpha_i)\}$, which is
\begin{align*}
\sum_{i=1}^3 E_{\ell-1}(\alpha_i) &= 3H_{\ell-1} - 2 G_{\ell-1} \sigma_1 + \sum_{i=1}^3 \alpha_i^2
= 3\bigg( \frac{B}{A} + 4 G_{\ell-1}^2  \bigg) - 0 + (\sigma_1^2 - 2\sigma_2)
\\
&= 3\frac{B}{A} + 12G_{\ell-1}^2 + \bigg( 0 - 2\frac{B}{A} \bigg)
= \frac{B}{A} + 12G_{\ell-1}^2 = s_{\ell,1}.
\end{align*}

Similarly, the $z$-coefficient of $q(z)$ is $\sigma_2(\{E_{\ell-1}(\alpha_i)\}$. Writing out
$E_{\ell-1}(\alpha_i)=H_{\ell-1} - 2G_{\ell-1} \alpha_i + \alpha_i^2$ for each $i=1,2,3$, then, this coefficient is
\begin{align*}
3H_{\ell-1}^2 & - 4G_{\ell-1}H_{\ell-1} \sigma_1 + 2H_{\ell-1} \sum_{i=1}^3 \alpha_i^2
+4G_{\ell-1}^2\sigma_2 - 2G_{\ell-1}\sum_{i\neq j} \alpha_i^2 \alpha_j + \sum_{i<j}^3 \alpha_i^2 \alpha_j^2
\\
&= 3H_{\ell-1}^2 - 0-4\frac{B}{A} H_{\ell-1} 
+ 4G_{\ell-1}^2 \cdot \frac{B}{A} - 2 G_{\ell-1} \cdot 3\frac{(1-x_0)}{A} + \frac{B^2}{A^2}
\\
&= 3H_{\ell-1}\bigg( H_{\ell-1} - \frac{B}{A} \bigg) - \frac{6}{A} G_{\ell-1} (1-x_0)
= -\frac{6}{A} G_{\ell-1} \big( -2AG_{\ell-1} H_{\ell-1} + 1 -x_0 \big),
\end{align*}
where we have used formula~\eqref{eq:Hident} and the identities
\[ \sum_{i\neq j} \alpha_i^2 \alpha_j = \sigma_1 \sigma_2 - 3\sigma_3 = 3\frac{(1-x_0)}{A}
\quad\text{and}\quad
\sum_{i<j}^3 \alpha_i^2 \alpha_j^2 = \sigma_2^2 - 2\sigma_1\sigma_3 = \frac{B^2}{A^2}.\]
In addition, we have
\begin{align*}
-2AG_{\ell-1} H_{\ell-1} + 1 &= -8AG_{\ell-1}^3 -2BG_{\ell-1} + 1
= AG_{\ell-1}^3 + B G_{\ell-1} + 1 - BF_{\ell-1}^2 G_{\ell-1}
\\
&= G_{\ell} = F_{\ell} \gamma + G_{\ell} = f^{\ell}(\gamma),
\end{align*}
where the first equality is by formula~\eqref{eq:Hident}, the second
is by~\eqref{eq:ellcond}, the third is by~\eqref{eq:Giter}, and the fourth is because $F_{\ell}=0$.
Thus, we have shown that the $z$-coefficient of $q(z)$ is
\[ -\frac{6}{A} G_{\ell-1} \big( f^{\ell}(\gamma) - x_0 \big) = s_{\ell,2}. \]

Next, the negative of the constant term of $q(z)$ is $\sigma_3(\{E_{\ell-1}(\alpha_i)\}$, which is
\[
\prod_{i=1}^3 \big( f^{\ell-1}(\gamma)-\alpha_i \big) \big( f^{\ell-1}(-\gamma)-\alpha_i\big)
= \frac{1}{A^2} \big( f^{\ell}(\gamma)-x_0 \big)^2 = s_{\ell,3},
\]
since $f(z) = A\prod_{i=1}^3 (z-\alpha_i)$ and $f^{\ell}(-\gamma)=f^{\ell}(\gamma)$.

Finally, combining our formulas for $s_{\ell,1},s_{\ell,2},s_{\ell,3}$ with equation~\eqref{eq:ellcond} gives
\[ s_{\ell,2}^2 - 4s_{\ell,1}s_{\ell,3} =
\frac{4}{A^3}\big( f^\ell(\gamma) - x_0 \big)^2\big( 9 A G_{\ell-1}^2 - B - 12 A G_{\ell-1}^2 \big)
= -\frac{4B}{3A^3}F_{\ell-1}^2\big( f^\ell(\gamma) - x_0 \big)^2.
\qedhere \]
\end{proof}

\begin{prop}
\label{prop:quartic}
With notation as in Proposition~\ref{prop:alphacubic},
let $\delta_i := \sqrt{E_{\ell-1}(\alpha_i)}$ for each $i=1,2,3$,
and define $\theta_1,\theta_2,\theta_3,\theta_4\in\Kbar$ be the four values of
\[ \pm \delta_1 \delta_2 \pm \delta_1 \delta_3 \pm \delta_2 \delta_3 \]
for which an even number of the $\pm$ signs are $-$. Then
\[ \prod_{i=1}^4 (z- \theta_i) = z^4 -2 s_{\ell,2} z^2 - 8s_{\ell,3} z + (s_{\ell,2}^2 - 4s_{\ell,1}s_{\ell,3}) \]
where
\[ \prod_{i=1}^3 \big(z- E_{\ell-1}(\alpha_i)\big) =
z^3 - s_{\ell,1} z^2 + s_{\ell,2} z - s_{\ell,3} .\]
\end{prop}

\begin{proof}
This is another brute-force calculation.
Writing $\theta_1=\delta_1 \delta_2 + \delta_1 \delta_3 + \delta_2 \delta_3$
and $\theta_2=\delta_1 \delta_2 - \delta_1 \delta_3 - \delta_2 \delta_3$, we have
\begin{equation}
\label{eq:thetasum}
\theta_1 + \theta_2 = 2\delta_1\delta_2
\quad\text{and}\quad \theta_3 + \theta_4 = -2\delta_1\delta_2,
\end{equation}
and similarly for other sums $\theta_i+\theta_j$.

It follows that
\[ \sigma_1(\{\theta_i\}) = \sum_{i=1}^4 \theta_i = 0, \]
so the $z^3$-coefficient of the desired quartic polynomial is $0$. In addition,
\begin{align*}
\sigma_2(\{\theta_i\}) &= \sum_{i<j} \theta_i \theta_j
= \frac{1}{2} (\theta_1 + \theta_2)(\theta_3 + \theta_4) +
\frac{1}{2}(\theta_1 + \theta_3)(\theta_2 + \theta_4) +
\frac{1}{2}(\theta_1 + \theta_4)(\theta_2 + \theta_3)
\\
&=-2\delta_1^2\delta_2^2 -2\delta_1^2\delta_3^2 -2\delta_2^2\delta_3^2
= -2 \sigma_2\big( \big\{ E_{\ell-1}(\alpha_i) \big\} \big) = -2s_{\ell,2},
\end{align*}
since $\delta_i^2 = E_{\ell-1}(\alpha_i)$. Thus,
the $z^2$-coefficient of the quartic is $-2s_{\ell,2}$.

We also have
\[ \theta_1 \theta_2 = \delta_1^2\delta_2^2 - \delta_1^2\delta_3^2 - \delta_2^2\delta_3^2 - 2\delta_1\delta_2\delta_3^2
\quad\text{and}\quad
\theta_3 \theta_4 = \delta_1^2\delta_2^2 - \delta_1^2\delta_3^2 - \delta_2^2\delta_3^2 + 2\delta_1\delta_2\delta_3^2.\]
Combined with equations~\eqref{eq:thetasum}, it follows that
\begin{align*}
\sigma_3(\{\theta_i\}) &= (\theta_1+\theta_2)\theta_3\theta_4 + \theta_1\theta_2(\theta_3+\theta_4)
= 2\delta_1\delta_2\big( \theta_3\theta_4 - \theta_1 \theta_2 \big)
\\
&= 2\delta_1\delta_2\big( - 2\delta_1\delta_2\delta_3^2 -  2\delta_1\delta_2\delta_3^2)
= - 8 \delta_1^2\delta_2^2\delta_3^2= -8\sigma_3\big( \big\{ E_{\ell-1}(\alpha_i) \big\} \big)=-8s_{\ell,3}.
\end{align*}
It also follows that
\begin{align*}
\sigma_4(\{\theta_i\}) &= (\theta_1\theta_2)(\theta_3\theta_4)
= \big(\delta_1^2\delta_2^2 - \delta_1^2\delta_3^2 - \delta_2^2\delta_3^2 + 2\delta_1\delta_2\delta_3^2\big)
\big(\delta_1^2\delta_2^2 - \delta_1^2\delta_3^2 - \delta_2^2\delta_3^2 - 2\delta_1\delta_2\delta_3^2\big)
\\
&=\big(\delta_1^2\delta_2^2 - \delta_1^2\delta_3^2 - \delta_2^2\delta_3^2\big)^2 -4 \delta_1^2\delta_2^2\delta_3^4
\\
&=\big(\delta_1^2\delta_2^2 + \delta_1^2\delta_3^2 + \delta_2^2\delta_3^2\big)^2
-4 \delta_1^4\delta_2^2\delta_3^2 -4 \delta_1^2\delta_2^4\delta_3^2 -4 \delta_1^2\delta_2^2\delta_3^4
\\
&=\sigma_2\big( \big\{ E_{\ell-1}(\alpha_i) \big\} \big)^2
-4 \sigma_1\big( \big\{ E_{\ell-1}(\alpha_i) \big\} \big) \sigma_3\big( \big\{ E_{\ell-1}(\alpha_i) \big\} \big)
=s_{\ell,2}^2 - 4s_{\ell,1}s_{\ell,3},
\end{align*}
giving the desired values for the $z$-coefficient and constant term of the quartic.
\end{proof}

\section{Proving Theorem~\ref{thm:main2}}
\label{sec:galproof}

Before proving our second main theorem, we need a few more results
to connect the polynomials of Section~\ref{sec:tech} to the group-theoretic
results of Section~\ref{sec:surj}.

\subsection{Lemmas on valuations}
\label{ssec:morelemmas}
Throughout Section~\ref{ssec:morelemmas}, we set the following notation.

\begin{tabbing}
\hspace{8mm} \= \hspace{15mm} \=  \kill
\> $\calO_L$: \> a Dedekind domain with field of fractions $L$
with algebraic closure $\Lbar$ \\
\> $M_L^0$: \> the set of non-archimedean places of $L$ \\
\> $f$: \> a cubic polynomial $f(z)=Az^3+Bz+1 \in L[z]$ \\
\> $x_0$: \> an element of $L$, to serve as the root of our preimage tree \\
\> $L_n$: \> for each $n\geq 0$, the extension field $L_n:=L(f^{-n}(x_0))$ \\
\> $E_n(t)$: \> for each $n\geq 1$, the quantity $(f^n(\gamma_1)-t)(f^n(\gamma_2)-t)$, \\
\> \> where $\gamma_1,\gamma_2\in\Lbar$ are the two critical points of $f$ in $\Lbar$ \\
\> $\tilde{E}_n(t)$: \> for $n\geq 1$ for which $f^n(\gamma_1)\in L$,
the quantity $f^n(\gamma_1)- t$ \\
\> $C_n$: \> for each $n\geq 1$, the quantity $(f^n(\gamma_1)-f^n(\gamma_2))^2\in L$
\end{tabbing}
The quantity $E_n(t)$ here coincides with the $E_n(t)$ of equation~\eqref{eq:Endef},
and $C_n$ coincidies with the $C_n$ of equation~\eqref{eq:BCdef}.
In addition, recall from Remark~\ref{rem:BC} that $C_n = -4BF_n^2 / (3A)\in L$,
where $F_n$ is as in equation~\eqref{eq:FGdef}.
Note that for any particular choice of $t\in L$, we have $E_n(t)\in L$.
Moreover, as stated in Theorem~\ref{thm:main1},
if the two critical points collide at the $\ell$-th iteration,
then for $n\geq\ell$, we have both $\tilde{E}_n(t)\in L$ and $C_n=0$.

\begin{lemma}
\label{lem:inherit1}
With notation as above, let $m,n\geq 1$ be integers,
let $\alpha\in f^{-m}(x_0)$, and let $v\in M_L^0$ such that $v(A)=0$.
\begin{enumerate}
\item
If $v(E_{m+n}(x_0))$ is odd, then there is some $w\in M_{L(\alpha)}^0$ such that $w(E_{n}(\alpha))$ is odd.
\item
If $f^n(\gamma_1)\in L$, and if $v(\tilde{E}_{m+n}(x_0))$ is odd,
then there is some $w\in M_{L(\alpha)}^0$ such that $w(\tilde{E}_{n}(\alpha))$ is odd.
\end{enumerate}
\end{lemma}

\begin{proof}
We prove the first statement; the second is similar.
We have
\[ f^m(z) -x_0 = A^{(3^m-1)/2} \prod_{\beta} (z-\beta), \]
where the product is over all $\beta\in f^{-m}(x_0)$.
Substituting $f^n(\gamma_1)$ and $f^n(\gamma_2)$ for $z$, we have
\begin{align*}
A^{3^m-1}\prod_{\beta} E_n(\beta)
&= A^{3^m-1}\prod_{\beta} \big(f^n(\gamma_1)-\beta\big) \big(f^n(\gamma_2)-\beta\big) \\
&= \big(f^{m+n}(\gamma_1)-x_0\big) \big(f^{m+n}(\gamma_2)-x_0\big) = E_{m+n}(x_0) .
\end{align*}
Therefore, recalling that $|A|_v=1$, it follows that
\[ \sum_{w|v} e_w f_w \log |E_n(\alpha)|_w
=\log \bigg| \prod_{\beta} E_n(\beta) \bigg|_{v}
= \log \big| E_{m+n}(x_0) \big|_{v}, \]
where the sum is over places $w\in M_{L(\alpha)}^0$ lying over $v$,
and where $e_w$ and $f_w$ denote the associated ramification index
and residue field extension degree, respectively.
(The first equality above is a standard norm relation for absolute values;
see, for example, Theorem~I.4.5 of \cite{Lang}.)
Thus, if $\pi_v\in L$ and $\pi_w\in L(\alpha)$ are uniformizers for $v$ and $w$, respectively, then
\begin{align*}
v\big(E_{m+n}(x_0)\big) &= \frac{\log \big| E_{m+n}(x_0) \big|_{v}}{\log|\pi_v|_v} 
= \sum_{w|v} e_w f_w \cdot \frac{\log \big| E_{n}(\alpha) \big|_{w}}{\log|\pi_v|_v} \\
&= \sum_{w|v} \frac{\log|\pi_v|_v}{\log|\pi_w|_w} \cdot f_w \cdot
\frac{\log \big| E_n(\alpha) \big|_{w}}{\log|\pi_v|_v}
= \sum_{w|v} f_w \cdot w\big(E_n(\alpha)\big).
\end{align*}
The right side of this equation is a sum of integers, and by hypothesis the left side is odd.
Thus, there must be some $w|v$ such that
$f_w \cdot w(E_n(\alpha))$, and hence $w(E_{n}(\alpha))$, is odd.
\end{proof}

The following result is essentially Proposition~2.3 of \cite{Looper}, 
which applies to polynomials like our $f(z)=Az^3+Bz+1$ that have only three nonzero terms;
see also Proposition~3.4 of \cite{BenJuu}.

\begin{lemma}
\label{lem:looper}
With notation as at the start of this section, let $n\geq 1$ be an integer, and suppose that there are places
$v_1,\ldots,v_n\in M_L^0$ such that for each $j=1,\ldots,n$,
\begin{enumerate}%
\item $v_j(A)=v_j(B)=v_j(C_j)=v_j(6)=0\leq v_j(x_0)$,
\item $v_j(E_i(x_0))=0$ for all $1\leq i \leq j-1$, and
\item $v_j(E_j(x_0))$ is odd.
\end{enumerate}
Suppose further that there is a non-archimedean place $u\in M_L^0$ such that $u(A)=0\leq u(B)$,
and $u(x_0)$ is negative and prime to $3$.
Then $\Gal(L_n/L)\cong \Aut(T_{3,n})$.
\end{lemma}

\begin{proof}
\textbf{Step 1}. By the assumptions on the place $u$,
for any $j\geq 1$, the iterate $f^j(z)$ has $u$-integral coefficients,
and its lead coefficient is a power of $A$ and hence is a $u$-unit.
By the assumption on $u(x_0)$, it follows that 
the $u$-adic Newton polygon of $f^j(z)-x_0$ is a single segment of slope $-u(x_0)/3^j$.
Thus, for any root $\beta\in f^{-j}(x_0)$, the field $L(\beta)$ is totally ramified over $u$,
and in particular $f^j(z)-x_0$ is irreducible over $L$.

\medskip

\textbf{Step 2}.
Proceeding inductively up the tree,
for any $j=1,2,\ldots,n$, suppose that we already have $\Gal(L_{j-1}/L)\cong \Aut(T_{3,j-1})$.
(Note that this supposition is trivial for $j=1$.)
It suffices to prove that $\Gal(L_j/L)\cong \Aut(T_{3,j})$
under this assumption.

For each $\alpha\in f^{-(j-1)}(x_0)$, define $\Lambda_\alpha$ to be the extension of $L(\alpha)$
formed by adjoining all three elements of $f^{-1}(\alpha)$. Further define $\widehat{\Lambda}_{\alpha}$
to be the compositum $\widehat{\Lambda}_{\alpha} := \prod_{\beta\neq\alpha} \Lambda_{\beta}$,
where this product is over all $\beta\in f^{-(j-1)}(x_0)\smallsetminus\{\alpha\}$.
Note that $\widehat{\Lambda}_{\alpha}$ contains every $\beta\in f^{-(j-1)}(x_0)\smallsetminus\{\alpha\}$
and hence also contains $\alpha$. In particular, $\widehat{\Lambda}_{\alpha}$ contains
the field $L_{j-1} = L(f^{-(j-1)}(x_0))$. See Figure~\ref{fig:Ltower}.

\begin{figure}
\begin{tikzpicture}
\draw (0.1,.2) -- (0.7,.55);
\draw (1.4,1.1) -- (2,1.45);
\draw (2.55,1.8) -- (3.3,2.25);
\draw (.3,2.1) -- (1.0,2.5);
\draw (1.65,3) -- (2.5,3.5);
\draw (0.9,1.1) -- (0.4,1.6);
\draw (2.27,1.9) -- (1.7,2.45);
\draw (3.45,2.75) -- (2.7,3.5);
\node (L) at (-0.1,0.1) {\small $L$};
\node (La) at (1.1,0.8) {\small $L(\alpha)$};
\node (Lj1) at (2.5,1.6) {\small $L_{j-1}$};
\node (LAh) at (3.7,2.5) {\small $\hat{\Lambda}_{\alpha}$};
\node (LA) at (.3,1.9) {\small $\Lambda_{\alpha}$};
\node (LALj) at (1.5,2.8) {\small $\Lambda_{\alpha} L_{j-1}$};
\node (Lj) at (2.7,3.8) {\small $\Lambda_{\alpha} \hat{\Lambda}_{\alpha} = L_j$};
\end{tikzpicture}
\caption{The fields in the proof of Lemma~\ref{lem:looper}.}
\label{fig:Ltower}
\end{figure}
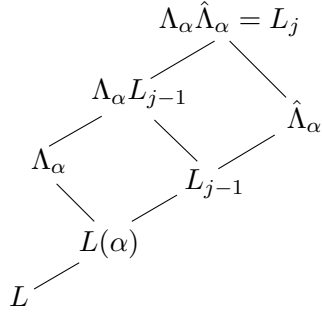

Fix $\alpha\in f^{-(j-1)}(x_0)$.
By Lemma~\ref{lem:inherit1}(1) and our hypothesis that $v_j(E_j(x_0))$ is odd,
there is some $w_j\in M_{L(\alpha)}^0$ such that $w_j(E_1(\alpha))$ is odd.
We claim that $w_j$ does not ramify in $\widehat{\Lambda}_{\alpha}$.
We will prove this claim in Step~4; first we turn to its consequences.

\medskip

\textbf{Step 3}. Assume the claim of Step~2.
Because $w_j(E_1(\alpha))$ is odd with $w_j(3)=w_j(A)=0$,
it follows from equation~\eqref{eq:cubiciter} that $w_j(\Delta(f(z)-\alpha))$ is odd,
and hence $w_j$ must ramify in $\Lambda_{\alpha}$.
By Theorem~2.1 of \cite{Looper}, the inertia group of
$\Lambda_{\alpha} / L(\alpha)$ at $w_j$
is generated by a transposition of two elements of $f^{-1}(\alpha)$.
(See also Lemma~2.2 of \cite{BenJuu}; this fact about inertia relies on the fact that $f(z)-\alpha$
is a trinomial.) Since $w_j$ does not ramify in $\widehat{\Lambda}_{\alpha}$ by the claim of Step~2,
it follows that the inertia group of $L_j = \Lambda_{\alpha} \widehat{\Lambda}_{\alpha}$ over $L(\alpha)$
at $w_j$ also is generated by a transposition of two elements of $f^{-1}(\alpha)$.

Thus, $\Gal(L_j/\widehat{\Lambda}_{\alpha})$
contains an element $\tau$ that acts as a transposition of two elements $\alpha'_0,\alpha'_1$ of $f^{-1}(\alpha)$,
and which necessarily fixes all elements of $f^{-j}(x_0)\smallsetminus \{\alpha'_0,\alpha'_1\}$,
including the third point $\alpha'_2$ of $f^{-1}(\alpha)$.
The larger group $\Gal(L_j/L_{j-1})$ therefore also contains $\tau$.
Since $\Gal(L_j/L)$ acts transitively on $f^{-j}(x_0)$ by Step~1, we may conjugate $\tau$ by
some $\sigma\in\Gal(L_j/L)$ for which $\sigma(\alpha'_0)=\alpha'_2$, to obtain $\tau'\in\Gal(L_j/L_{j-1})$
that transposes $\alpha'_1$ and $\alpha'_2$ while fixing all other elements of $f^{-j}(x_0)$.
Together, $\tau$ and $\tau'$ generate the full group $S_3$ of permutations of $f^{-1}(\alpha)$
that fix $f^{-j}(x_0)\smallsetminus f^{-1}(\alpha)$.

Again by Step~1, 
$\Gal(L_{j-1}/L)$ acts transitively on $f^{-(j-1)}(x_0)$, so we may conjugate the above copy of $S_3$
to act on $f^{-1}(\beta)$ for any $\beta\in f^{-(j-1)}(x_0)$.
It follows that $\Gal(L_j/L_{j-1})\cong (S_3)^{3^{j-1}}$, which is the full subgroup of $\Aut(T_{3,j})$
fixing level $j-1$ of the tree. Since $\Gal(L_{j-1}/L)$ is all of $\Aut(T_{3,j-1})$,
it follows that $\Gal(L_j/L)\cong \Aut(T_{3,j})$, completing our induction on $j$.

\medskip

\textbf{Step 4}. It remains to prove the claim of Step~2.
Note that $v_j(\Delta(f^{j-1}(z)-x_0))=0$, since the only factors of $\Delta(f^{j-1}(z)-x_0)$
in equation~\eqref{eq:cubiciter} involve $3$, $A$, and $E_1(x_0),\ldots,E_{j-1}(x_0)$.
Therefore $w_j(\Delta(f^{j-1}(z)-x_0))=0$ as well, and hence $w_j$ does not ramify in $L_{j-1}$.
Let $W_j$ be a place of $L_{j-1}$ lying over $w_j$.
It suffices to show that $W_j$ does not ramify in $\widehat{\Lambda}_{\alpha}$.

Suppose not, i.e., suppose that there is some $\beta\in f^{-(j-1)}(x_0)\smallsetminus \{\alpha\}$
such that $W_j$ ramifies in the compositum $L_{j-1} \Lambda_\beta$.
Then $W_j(\Delta(f(z)-\beta))>0$, but we still have $W_j(3)=W_j(A)=0$, and therefore
$W_j(E_1(\beta))>0$. Since we already had $w_j(E_1(\alpha))$ odd (and hence positive),
it follows that each of $\alpha$ and $\beta$ is congruent modulo $\fP$
to one of the critical values $f(\pm\gamma)$ of $f$,
where $\fP$ is the prime ideal of $\calO_{L_{j-1}}$ corresponding to $W_j$.
More precisely, even if $\gamma$ is not $L_{j-1}$-rational, 
we mean that the quadratic polynomial $E_1(t)$ factors modulo~$\fP$,
and we denote its two roots in $\calO_{L_{j-1}}/\fP$ by $f(\gamma)$ and $f(-\gamma)$
for the remainder of the proof.
Without loss, suppose $\alpha\equiv f(\gamma) \pmod{\fP}$.
We consider two cases.

If $\beta\equiv f(\gamma) \pmod{\fP}$, then
$W_j(\alpha-\beta)>0$. But then $W_j(\Delta(f^{j-1}(z)-x_0))>0$ as well, and hence $v_j(\Delta(f^{j-1}(z)-x_0))>0$,
since $(\alpha-\beta)^2$ is a factor of this discriminant. However, this same discriminant
is a product of powers of $3$, $A$, and $E_i(x_0)$ for $1\leq i\leq j-1$,
all of which have $v_j(\cdot)=0$ by our hypotheses. Thus, we have a contradiction.

Otherwise, we have $\beta\equiv f(-\gamma) \pmod{\fP}$. Then
\[ C_j = \big(f^j(\gamma)-f^j(-\gamma)\big)^2
\equiv \big( f^{j-1}(\alpha) - f^{j-1}(\beta) \big)^2 = (x_0-x_0)^2 = 0 \pmod{\fP}. \]
That is, $W_j(C_j)>0$, and hence $v_j(C_j)=0$, again contradicting our hypotheses.
\end{proof}

Our next result gives a criterion for our Galois groups at level $\ell$ not to be contained in
any of the four subgroups $H_{\ell,1},H_{\ell,2},H_{\ell,3},H_{\ell,4}$ of $Q_{\ell,\ell}$
from Section~\ref{ssec:surjpre}.

\begin{lemma}
\label{lem:ellbump}
With notation as at the start of this section,
suppose that the two critical points $\gamma_1,\gamma_2$ of $f$
collide at the $\ell$-th iterate, for some $\ell\geq 2$. Suppose further that
there is a place $v_{\ell}\in M_L^0$ for which
\begin{enumerate}
\item $v_\ell(A)=v_\ell(B)=v_{\ell}(C_{\ell-1}) = v_\ell(6)=0\leq v_\ell(x_0)$,
\item $v_\ell(E_i(x_0))=0$ for all $1\leq i \leq \ell-1$, and
\item $v_\ell(\tilde{E}_\ell(x_0))$ is odd.
\end{enumerate}
Then $\Gal(L_\ell/L_{\ell-1})$ is not contained in any of the subgroups
$H_{\ell,1},H_{\ell,2},H_{\ell,3},H_{\ell,4}$ of $Q_{\ell,\ell}$.
\end{lemma}

\begin{proof}
Write $f^{-1}(x_0)=\{\alpha_1,\alpha_2,\alpha_3\}$.
As in Proposition~\ref{prop:quartic}, define $\delta_i:=\sqrt{E_{\ell-1}(\alpha_i)}$ for each $i=1,2,3$,
and $\theta_1,\theta_2,\theta_3,\theta_4$
to be the four values of $\pm\delta_1\delta_2 \pm\delta_1\delta_3 \pm\delta_2\delta_3$
with an even number of $-$ signs.
Recalling that $C_{\ell-1}=-4BF_{\ell-1}^2/(3A)$,
Propositions~\ref{prop:alphacubic} and~\ref{prop:quartic}
together say that the quartic polynomial with roots $\theta_1,\theta_2,\theta_3,\theta_4$ is
\begin{equation}
\label{eq:quartnewt}
z^4 + \frac{12}{A} G_{\ell-1} \tilde{E}_{\ell}(x_0)z^2 - \frac{8}{A^2} \big( \tilde{E}_{\ell}(x_0) \big)^2 z
+ \frac{1}{A^2} C_{\ell-1} \big( \tilde{E}_{\ell}(x_0) \big)^2 .
\end{equation}
By our hypotheses, we have $v_\ell(\tilde{E}_\ell(x_0))>0$, and
the Newton polygon of this polynomial at the valuation $v_{\ell}$
appears in Figure~\ref{fig:quartnewt}.
The slope of the only segment of the polygon is $v_\ell(\tilde{E}_\ell(x_0))/2$,
which is not an integer.
(The coefficient of $z^2$ may have valuation greater than the value
$v_\ell(\tilde{E}_\ell(x_0))$ shown in Figure~\ref{fig:quartnewt},
but this would not change the polygon itself.)

\begin{figure}
\begin{tikzpicture}
\draw[->] (-.3,0) -- (5,0);
\draw[->] (0,-.3) -- (0,2.7);
\path[draw] (-.2,1) -- (.2,1);
\path[draw] (-.2,2) -- (.2,2);
\path[draw] (1,-.2) -- (1,.2);x
\path[draw] (2,-.2) -- (2,.2);
\path[draw] (3,-.2) -- (3,.2);
\path[draw] (4,-.2) -- (4,.2);
\path[fill] (0,2) circle (0.1);
\path[fill] (1,2) circle (0.1);
\path[fill] (2,1) circle (0.1);
\path[fill] (4,0) circle (0.1);
\path[draw,dashed] (0,2) -- (4,0);
\node (x1) at (1,-0.4) {\footnotesize $1$};
\node (x2) at (2,-0.4) {\footnotesize $2$};
\node (x3) at (3,-0.4) {\footnotesize $3$};
\node (x4) at (4,-0.4) {\footnotesize $4$};
\node (v1) at (-1.1,1) {\footnotesize $v_{\ell}(\tilde{E}_\ell(x_0))$};
\node (v2) at (-1.15,2) {\footnotesize $2v_{\ell}(\tilde{E}_\ell(x_0))$};
\end{tikzpicture}
\caption{The Newton polygon of the quartic polynomial~\eqref{eq:quartnewt}, where $\tilde{E}=\tilde{E}_{\ell}(x_0)$.}
\label{fig:quartnewt}
\end{figure}
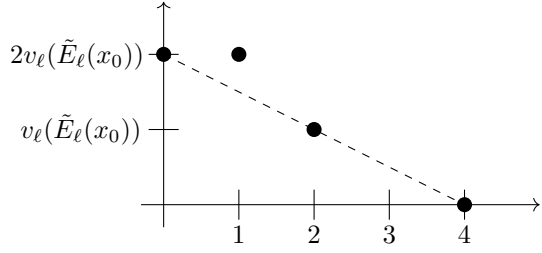

Because $v_\ell(E_i(x_0))=0$ for all $1\leq i \leq \ell-1$, it follows
that $v_{\ell}(\Delta(f^{\ell-1}-x_0))=0$ by equation~\eqref{eq:cubiciter},
and hence $v_{\ell}$ does not ramify in $L_{\ell-1}$.
Thus, because of the non-integral slope of the Newton polygon, 
the polynomial~\eqref{eq:quartnewt} has no $L_{\ell-1}$-rational roots,
i.e., $\theta_1,\theta_2,\theta_3,\theta_4\not\in L_{\ell-1}$.

If $\Gal(L_{\ell}/L_{\ell-1})$ is contained in some $H_{\ell,j}$,
then there would be a labeling on $T_{3,\ell}$ such that
every $\sigma\in\Gal(L_{\ell}/L_{\ell-1})$ satisfies
$\sgn_{\ell-1}(\sigma,\alpha_1) = \sgn_{\ell-1}(\sigma,\alpha_2) = \sgn_{\ell-1}(\sigma,\alpha_3)$.
That is,
\[ \frac{\sigma(\sqrt{\delta_1})}{\sqrt{\delta_1}} =
\frac{\sigma(\sqrt{\delta_2})}{\sqrt{\delta_2}} =\frac{\sigma(\sqrt{\delta_3})}{\sqrt{\delta_3}}, \]
and hence $\sigma(\theta_j)=\theta_j$ for every $\sigma\in\Gal(L_{\ell}/L_{\ell-1})$.
But then $\theta_j\in L_{\ell-1}$, contradicting the previous paragraph and proving the desired result.
\end{proof}

The hypotheses of Theorem~\ref{thm:main2}, as well as those
of Lemmas~\ref{lem:looper} and~\ref{lem:ellbump},
inspire the following terminology. For a field $L$ and root point $x_0\in L$ as above,
we will say that the pair $(L,x_0)$ satisfies condition~$(\dagger)$ if
there are places $v_1,v_2,\ldots\in M_L^0$ such that for every $n\geq 1$:
\[ (\dagger) \; \begin{cases}
v_n(A)=v_n(B)=v_n(6)=0\leq v_n(x_0), \\
v_n(C_j)=0 \text{ for all } 1\leq j \leq \min\{\ell-1,n\}, \\
v_n(E_i(x_0))=0 \text{ for all } 1\leq i \leq n-1, \\
\text{if } n\leq \ell-1, \text{ then } v_n(E_n(x_0)) \text{ is odd, and} \\
\text{if } n\geq \ell, \text{ then } v_n(\tilde{E}_n(x_0)) \text{ is odd.} \\
\end{cases} \]



\begin{lemma}
\label{lem:inherit}
With notation as at the start of this section, 
suppose that the two critical points $\gamma_1,\gamma_2$ of $f$
collide at the $\ell$-th iterate, for some $\ell\geq 2$.
Suppose further that $(L,x_0)$ satisfies condition~$(\dagger)$.
Let $\alpha\in f^{-1}(x_0)\in\Lbar$. Then $(L(\alpha),\alpha)$ also satisfies condition~$(\dagger)$.
\end{lemma}

\begin{proof}
For each $n\geq 1$, we will choose $w_n\in M_{L(\alpha)}^0$ lying above $v_{n+1}$,
and prove that $w_1,w_2,\ldots$ satisfy the five conditions of~$(\dagger)$ with $w_n$ in place of $v_n$
and $\alpha$ in place of $x_0$.
Since $w_n|v_{n+1}$, we will automatically have $w_n(A)=w_n(B)=w_n(6)=0$,
as well as the condition that $w_n(C_j)=0$ for $1\leq j\leq \min\{\ell-1,n\}$.
In addition, because
\[ Az^3+Bz + 1-x_0 = f(z)-x_0=A\prod_{j=1}^3 (z-\alpha_j),\]
where $\alpha=\alpha_1$, we will have $w_n(\alpha)\geq 0$ for all $n$.
Furthermore, for any $i\geq 1$, substituting $z=f^{i}(\gamma)$ into this formula yields
\[ E_{i+1}(x_0) = A^{2} \prod_{j=1}^3 E_{i}(\alpha_j) .\]
Recalling from the third condition of~$(\dagger)$ for $(L,x_0)$
that $v_{n+1}(E_{i+1}(x_0))=0$ for $1\leq i+1\leq n$,
then because all the terms in these formulas are $w_n$-integral,
it follows that $w_n(E_i(\alpha))=0$ for all $1\leq i \leq n - 1$.
It remains to choose $w_n|v_{n+1}$ to ensure that the last two
conditions of~$(\dagger)$ hold for $(L(\alpha),\alpha)$.

For $1\leq n\leq \ell-2$, applying Lemma~\ref{lem:inherit1}(1) to $v_{n+1}$ with $m=1$
yields $w_n\in M_{L(\alpha)}^0$ with $w_n | v_{n+1}$ such that $w_n(E_n(\alpha))$ is odd.
For $n\geq \ell$, applying Lemma~\ref{lem:inherit1}(2) to $v_{n+1}$ with $m=1$
yields $w_n\in M_{L(\alpha)}^0$ with $w_n | v_{n+1}$ such that $w_n(\tilde{E}_n(\alpha))$ is odd.

Finally, consider $n=\ell-1$. The field $L(\gamma_1)=L(\gamma_2)$ is unramified at $v_{\ell}$
because $\gamma_1,\gamma_2$ are roots of $z^2+B/(3A)$, which has discriminant
$-4B/(3A)$ of $v_{\ell}$-valuation zero. Thus, choosing $\tilde{v}_{\ell}\in M_{L(\gamma_1)}^0$
lying above $v_{\ell}$, we have $\tilde{v}_{\ell}(\tilde{E}_{\ell}(x_0))=v_\ell(\tilde{E}_{\ell}(x_0))$ is odd.
Since $\gamma_1,\gamma_2\in L(\gamma_1)$,
we may again apply Lemma~\ref{lem:inherit1}(2) to $\tilde{v}_{\ell}$ with $m=1$
to obtain $\tilde{w}_{\ell-1}\in M_{L(\alpha,\gamma_1)}^0$ with $\tilde{w}_{\ell-1} | \tilde{v}_{\ell}$
such that $\tilde{w}_{\ell-1}(\tilde{E}_{\ell-1}(\alpha))$ is odd; see Figure~\ref{fig:LagFields}.

\begin{figure}
\begin{tikzpicture}
\draw (0.2,.12) -- (0.9,.57);
\draw (-0.2,.12) -- (-0.9,.57);
\draw (0.9,1.25) -- (0.2,1.7);
\draw (-0.9,1.25) -- (-0.2,1.7);
\node (L) at (0,0) {\small $L$};
\node (La) at (1.1,0.9) {\small $L(\alpha)$};
\node (Lg) at (-1.1,0.9) {\small $L(\gamma_1)$};
\node (Lag) at (0,2.0) {\small $L(\alpha,\gamma_1)$};
\draw (4.2,.12) -- (4.9,.57);
\draw (3.8,.12) -- (3.1,.57);
\draw (4.9,1.25) -- (4.2,1.7);
\draw (3.1,1.25) -- (3.8,1.7);
\node (v) at (4,0) {\small $v_\ell$};
\node (w) at (5.1,0.9) {\small $w_{\ell-1}$};
\node (vt) at (2.9,0.9) {\small $\tilde{v}_\ell$};
\node (wt) at (4,2.0) {\small $\tilde{w}_{\ell-1}$};
\end{tikzpicture}
\caption{The fields and valuations  in the proof of Lemma~\ref{lem:inherit}.}
\label{fig:LagFields}
\end{figure}
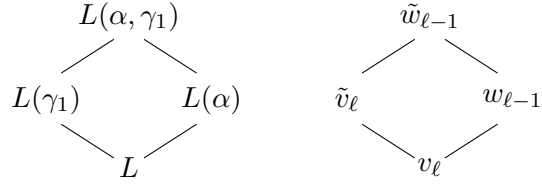

Because $\tilde{w}_{\ell-1}(f^{\ell-1}(\gamma_1)-\alpha)$ is odd, it is positive,
as all the terms involved are $\tilde{w}_{\ell-1}$-integral.
If $\tilde{w}_{\ell-1}(f^{\ell-1}(\gamma_2)-\alpha)$ were also positive,
then we would have
\[ \tilde{w}_{\ell-1}(C_{\ell-1}) 
= 2\tilde{w}_{\ell-1}\big(f^{\ell-1}(\gamma_1)-f^{\ell-1}(\gamma_1)\big) > 0, \]
contradicting the fact that $v_\ell(C_{\ell-1})=0$.
Thus, we must have 
$\tilde{w}_{\ell-1}(f^{\ell-1}(\gamma_2)-\alpha)=0$. It follows that 
\[ \tilde{w}_{\ell-1}\big( E_{\ell-1}(\alpha) \big)
= \tilde{w}_{\ell-1}\big( f^{\ell-1}(\gamma_1)-\alpha \big) + \tilde{w}_{\ell-1}\big( f^{\ell-1}(\gamma_2)-\alpha \big)
= \tilde{w}_{\ell-1}\big(\tilde{E}_{\ell-1}(\alpha)\big) + 0 \]
is odd. Let $w_{\ell-1}\in M_{L(\alpha)}^0$ be the place of $L(\alpha)$ lying below $\tilde{w}_{\ell-1}$.
The extension $L(\alpha,\gamma_1)/L(\alpha)$ is unramified at $w_{\ell-1}$, which lies above $v_{\ell}$,
for the same reason that $L(\gamma_1)/L$ was unramified at $v_\ell$.
Thus, $w_{\ell-1}( E_{\ell-1}(\alpha) )=\tilde{w}_{\ell-1}( E_{\ell-1}(\alpha) )$ is odd, as desired.
\end{proof}

\subsection{Proving Theorem~\ref{thm:main2}}
\label{ssec:endproof}

The foregoing lemmas provide us with all the tools we need to prove our second main result.

\begin{proof}[Proof of Theorem~\ref{thm:main2}]
By Proposition~\ref{prop:collide}, we may make a $K$-rational
change coordinates to assume that $f$ is of the form $f(z)=Az^3+Bz+1$.
By Remark~\ref{rem:BC}, the value of $B$ here coincides with that in 
equation~\eqref{eq:BCdef}.
Thus, the hypotheses of Theorem~\ref{thm:main2}
say that condition~$(\dagger)$ holds for the pair $(K,x_0)$.

Suppose first that $K$ contains $\sqrt{-3}$.
By inductively applying Lemma~\ref{lem:inherit}, for every $\alpha\in\Orb_f^{-}(x_0)$,
the pair $(K(\alpha),\alpha)$ also satisfies condition~$(\dagger)$.
Also for any such $\alpha$, if we pick any $\tilde{u}\in M_{K(\alpha)}^0$
lying over the place $u$ of Theorem~\ref{thm:main2},
say with $\alpha\in f^{-n}(x_0)$,
then because $u(x_0)$ is negative and prime to $3$,
the ramification index of $\tilde{u}$ over $u$ must be $3^n$;
that is, $\tilde{u}$ is totally ramified over $u$.
It follows that $\tilde{u}(\alpha)$ is also negative and prime to $3$.
In particular, for the field $L=K(\alpha)$ and root point $\alpha$,
the hypotheses of Lemma~\ref{lem:looper} apply
for each of $n=1,\ldots,\ell-1$,
and the hypotheses of Lemma~\ref{lem:ellbump} apply for $n=\ell$.

Therefore, for every $\alpha\in\Orb_f^{-}(x_0)$, writing
$K(\alpha)_n$ for the field $K(f^{-n}(\alpha))$,
repeated application of Lemma~\ref{lem:looper} shows that
$\Gal(K(\alpha)_{\ell-1} / K(\alpha))\cong \Aut(T_{3,\ell-1})$.
Similarly, Lemma~\ref{lem:ellbump} shows that $\Gal(K(\alpha)_{\ell}/K(\alpha)_{\ell-1})$
is not contained in any of the subgroups $H_{\ell,1},H_{\ell,2},H_{\ell,3},H_{\ell,4}$
of Section~\ref{ssec:surjpre}.

On the other hand, by Theorem~\ref{thm:main1}, for each $\alpha\in\Orb_f^{-}(x_0)$
and $n\geq 1$, we have $\Gal(K(\alpha)_n/K(\alpha))\subseteq Q_{\ell,n}$.
Therefore, by inductively applying Theorem~\ref{thm:gen}, it follows that
$\Gal(K(\alpha)_n/K(\alpha))$ is all of $Q_{\ell,n}$ for every $\alpha\in\Orb_f^{-}(x_0)$
and every $n\geq 1$. In particular, fixing $\alpha=x_0$ and taking the inverse limit with respect to $n$,
we have $G_{K,\infty}\cong Q_{\ell,\infty}$, as desired.

It remains to consider the case that $\sqrt{-3}\not\in K$.
Let $L:=K(\sqrt{-3})$, which is ramified at only those finite places of $K$ dividing $3$.
Since $v_n(3)=0$ for all the places $v_n$ in the statement of Theorem~\ref{thm:main2}, 
hypotheses (1)--(5) still hold for places $\tilde{v}_n\in M_L^0$ lying above each $v_n$.
In addition, whether or not the place $u$ ramifies in $L$,
any place $\tilde{u}\in M_L^0$ lying above $L$ still has $u(x_0)$ negative and prime to $3$,
since $[L:K]=2$.

Thus, the hypotheses of Theorem~\ref{thm:main2} apply to $L$ in place of $K$,
and therefore by the first portion of this proof, we have $\Gal(L_{\infty}/L)\cong Q_{\ell,\infty}$.
However, as noted in Step~2 of the proof of Theorem~\ref{thm:main1} in Section~\ref{ssec:pfthm1},
we have $\sqrt{-3}\in K_{\ell}\subseteq K_\infty$, and hence $L_{\infty}=K_{\infty}$.
Therefore, $G_{K,\infty}=\Gal(K_{\infty}/K)$ contains $\Gal(L_{\infty}/L)\cong Q_{\ell,\infty}$
as a subgroup of index $[L:K]=2$.
On the other hand, by Theorem~\ref{thm:main1}, $G_{K,\infty}$ is a subgroup of $\tilde{Q}_{\ell,\infty}$,
which contains $Q_{\ell,\infty}$ as a subgroup of index $2$, by Theorem~\ref{thm:QGroup}.
Hence, $G_{K,\infty}\cong\tilde{Q}_{\ell,\infty}$.
\end{proof}

\section{Examples}
\label{sec:ex}

\begin{example}
\label{ex:generic}
The choice of $A=33$ and $B=9$ satisfies the equation $81A + 27B - 4B^3=0$
noted just after equation~\eqref{eq:ellcond} in Section~\ref{ssec:explicit}, producing
the polynomial
\[ f(z) = 33z^3 + 9z+1, \]
whose critical points at $\pm\gamma=\pm\sqrt{-1/11}$ collide at iteration $\ell=2$.
Specifically, the critical orbit is
\[ \pm \gamma \mapsto 1 \pm 6\gamma
\mapsto -281
\mapsto -732207881
\mapsto -12954395051231033048301572681
\mapsto \cdots , \]
which, after the second term, is a rapidly decreasing sequence of negative integers.
Thus, for any choice of $x_0$, we have
\begin{align}
\label{eq:A33B9}
E_1(x_0) &= \big( 1+6\gamma - x_0)(1-6\gamma - x_0) = (1- x_0)^2 + \frac{36}{11} = x_0^2 - 2x_0 + \frac{47}{11}
\notag \\
\tilde{E}_2(x_0) &= -281 - x_0, \\
\tilde{E}_3(x_0) &= -732207881 - x_0, \notag \\
\tilde{E}_4(x_0) &= -12954395051231033048301572681 - x_0, \notag
\end{align}
and so on. We also have $C_1=(f(\gamma)-f(-\gamma))^2 = (12\gamma)^2 = -144/11$.

Let $K$ be the function field $K=k(t)$, where $k$ is a field of characteristic zero,
and let $x_0=t$. Choose $v_1\in M_K^0$ to be a place corresponding to
an irreducible factor of $E_1(t)=t^2-2t+47/11$,
and $v_n$ to be the place corresponding to the irreducible polynomial
$\tilde{E}_n(t)=f^n(\gamma)-t\in k[t]$ for each $n\geq 2$.
Then the places $v_1,v_2,\ldots\in M_K^0$ are all distinct,
and they satisfy hypotheses (1)--(5) of Theorem~\ref{thm:main2}.
We may further define $u\in M_K^0$ to be the negative degree valuation on $k[t]$,
so that $u(33)=u(9)=0$, and $u(t)=-1$ is negative and prime to $3$.
Thus, by Theorem~\ref{thm:main2}, the Galois group $G_{K,\infty}$
is all of $Q_{2,\infty}$ if $\sqrt{-3}\in k$,
or all of $\tilde{Q}_{2,\infty}$ otherwise.

More generally, for any $\ell\geq 2$ and any choice of $A,B$ satisfying
equation~\eqref{eq:ellcond} for which the forward orbit of $\gamma=\sqrt{-B/(3A)}$ is 
not preperiodic, the corresponding result should hold, that using $x_0=t\in K=k(t)$ as
the root point, the arboreal Galois group $G_{K,\infty}$
is either $Q_{\ell,\infty}$ if $\sqrt{-3}\in k$, or $\tilde{Q}_{\ell,\infty}$ otherwise.

In particular, suppose $k$ itself is the function field of the curve
\[ 9A G^2_{\ell-1}(A,B) + 3B - B F^2_{\ell-1}(A,B)=0\]
specified by equation~\eqref{eq:ellcond}.
Then it is reasonable to expect that the forward orbit
of $\gamma$ would indeed not be preperiodic, so that generically,
$G_{K,\infty}$ should be all of $Q_{\ell,\infty}$ or $\tilde{Q}_{\ell,\infty}$.
%
\end{example}

\begin{example}
\label{ex:ell2overQ}
Consider the map $f(z)=33z^3+9z+1$ of Example~\ref{ex:generic} with $\ell=2$, but this time
working over the field $K=\QQ$. 
With root point $x_0= -31/5$, we may choose $u\in M_{\QQ}^0$
to be the valuation at the prime $5$. Direct substitution into equation~\eqref{eq:A33B9}
shows
\[ E_1\bigg(-\frac{31}{5}\bigg) = \frac{36}{25 \cdot 11} \cdot 421,
\quad 
\tilde{E}_2\bigg(-\frac{31}{5}\bigg) = -\frac{6}{5} \cdot 229,
\quad
\tilde{E}_3\bigg(-\frac{31}{5}\bigg) = -\frac{6}{5} \cdot 401 \cdot 1521629, \]
and
\[ \tilde{E}_4\bigg(-\frac{31}{5}\bigg) = -\frac{6}{5} \cdot 43 \cdot 347651 \cdot 722144241378612874253 .\]
Thus, there is indeed a new prime $v_n$ satisfying conditions (1)--(5) of Theorem~\ref{thm:main2},
at least for $n=1,2,3,4$; and it seems reasonable to expect that the same is true for $n\geq 5$ as well.
That is, we certainly have $G_{\QQ,4} \cong \tilde{Q}_{2,4}$, and we expect that
$G_{\QQ,n} \cong \tilde{Q}_{2,n}$ for all $n\geq 1$, and hence that $G_{\QQ,\infty} \cong \tilde{Q}_{2,\infty}$.
Proving the existence of such an infinite sequence of primes, 
sometimes called primitive prime divisors, arises in other arboreal Galois problems,
such as in \cite{BriTuc,GNT,Hindes,JKetal}, for similar reasons. It is currently unclear how to solve
such problems in general without the use of the $abc$-Conjecture, Vojta's Conjecture, or the like.
\end{example}

\begin{example}
\label{ex:ell2overQbad}
Still with $K=\QQ$ and $f(z)=33z^3+9z+1$ and $\ell=2$,
now consider having the  root point be $x_0=-827/4$. We may choose $u\in M_{\QQ}^0$
to be the valuation at the prime $2$. Direct computation shows
\[ E_1\bigg(-\frac{827}{4}\bigg) = \frac{3^4 \cdot 93787}{2^4 \cdot 11},
\quad
\tilde{E}_2\bigg(-\frac{827}{4}\bigg) = \frac{3^4 \cdot 11}{2^2}, \]
and the quartic of Proposition~\ref{prop:quartic} is
\[ z^4 - 3^3 z^2 - \frac{3^4 \cdot 11^2}{2} z - 3^6 \cdot 11
= (z-18)\bigg( z^3 + 18 z^2 + 297z+\frac{891}{2} \bigg) \]
which has a $\QQ$-rational root. Thus, the Galois group $G_2=\Gal(K_2/\QQ)$
is not all of $\tilde{Q}_{2,2}=\Aut(T_{3,2})$, but rather is contained in the proper subgroup $H$
described in Section~\ref{ssec:surjpre}.


Finding special values for the root point $x_0$ as in this example requires some care.
After all, in light of Example~\ref{ex:generic} and Hilbert's irreducibility theorem, the set of
parameters $x_0\in\QQ$ for which the relevant quartic has a rational root
forms a thin set in the sense of Serre.
Thus, to find this example, instead of treating $x_0$ as an independent parameter and searching for values
that would make the quartic have a root, we instead treated $\tilde{E}_2$ as the independent parameter,
seeking values for which the quartic would have a root. To clear denominators, as well as knowing
that the prime-to-66 part of $\tilde{E}_2$ must be square, we further wrote $\tilde{E}_2=-33a^2$
and $z=aw$ for parameters $a$ and $w$, settling on $a=3/2$ and $w=12$, to obtain
values for $x_0=-281-\tilde{E}_2$ that are of relatively small arithmetic height.
\end{example}

\textbf{Acknowledgments}.
The authors gratefully acknowledge the support of NSF grant DMS-2101925;
these results grew out of a 2022 REU project at Amherst College funded through that grant.
The authors also thank Harris Daniels and David Zureick-Brown for helpful conversations.

\bibliographystyle{amsalpha}

\end{document}